\documentclass[10pt]{amsart}
\usepackage[utf8x]{inputenc}

\usepackage{dsfont}
\usepackage{amsmath}
\usepackage{amssymb}
\usepackage{graphicx}
\usepackage{amssymb}
\usepackage{amsfonts}
\usepackage{soul}
\usepackage{mathpazo}
\usepackage{color}
\usepackage{yfonts}
\usepackage{paralist}
\usepackage{stmaryrd}
\usepackage{amsxtra}
\usepackage{latexsym,mathrsfs}
\usepackage{tikz}

\usepackage{mathabx}

\usepackage{epsfig, enumerate}
\usepackage{color}
\usepackage{hyperref}

\newtheorem{theorem}{Theorem}
\newtheorem{lemma}[theorem]{Lemma}
\newtheorem{coro}[theorem]{Corollary}
\newtheorem{definition}{Definition}

\newtheorem{prop}[theorem]{Proposition}
\newtheorem{Remark}[theorem]{Remark}

\numberwithin{theorem}{section}
\numberwithin{equation}{section}
\usepackage{tikz}

\title[Non-contractible Periodic Orbits]{On Non-contractible Periodic Orbits and Bounded Deviations}

\author{Xiao-Chuan Liu}
\address[Liu]{Instituto de Matem\' atica e Estat\' istica da Universidade de S\~ao Paulo,
R. do Mat\~ ao, 1010 - Vila Universitaria, S\~ ao Paulo, Brasil} 
\email{lxc1984@gmail.com}

\author{F\' abio Armando Tal}
\address[Tal]{Instituto de Matem\' atica e Estat\' istica da Universidade de S\~ao Paulo,
R. do Mat\~ ao, 1010 - Vila Universitaria, S\~ ao Paulo, Brasil} 
\email{fabiotal@ime.usp.br}

\begin{document}
\maketitle{}
\begin{abstract}
We present a dichotomy for surface homeomorphisms in the isotopy class of the identity. 
We show that, in the absence of a degenerate fixed point set, either there exists a uniform bound on the diameter of orbits of non-wandering points for the lifted dynamics in the universal covering space, or the map has non-contractible periodic points. We then use this new tool to characterize the dynamics of area preserving homeomorphisms of the torus without non-contractible periodic points, showing that if the fixed point set is non-degenerate, then either the lifted dynamics is uniformly bounded, or the lifted map has a single strong irrational dynamical direction.
% it suffices to have non-unforFor any fixed totally irrational vector $(\alpha,\beta)$, let $\widetilde f$ be a lift of 
%a non-wandering homeomorphism $f$ on the two torus, whose rotation set 
%is a line segment from $(0,0)$ to $(\alpha,\beta)$. 
%We show that the displacement of all $\widetilde f$-forward orbits in the direction $-(\alpha,\beta)$ are uniformly bounded.
\end{abstract}

\section{Introduction}
%\marginpar{Changed the order of the first two paragraphs}
In this work we study the dynamical consequences of the existence or absence of non-contractible periodic points in conservative surface dynamics. 
This is a direction that has  been garnering increased attention in recent years, with new results due to mostly the increasingly developed field of Brouwer-homeomorphims like techniques, and that has drawn increased interest due not only to its applications in the study of some relevant classes of torus homeomorphisms, but also to its connection to symplectic dynamics, where the subject has been much more largely exploited 
(see for example ~\cite{gurel2013non} as well as references herein, see also ~\cite{orita2017existence}), albeit through very different techniques. 

Let $S$ denote a closed oriented surface, endowed with a metric of constant curvature, and let $\widetilde S$ be its universal covering, with the associated metric, which we simply denote $\| \cdot \|$.
Assume that we have a homeomorphism $f$ of $S$ in the isotopy class of the identity. Taking an isotopy $I=(f_t)_{t\in[0,1]}$ between the identity $f_0=Id$ and $f_1=f$, 
one can naturally associate to each point $x\in S$ 
a curve $\gamma_x:[0,1]\to S$, defined by $\gamma_x(t)=f_t(x)$. 
If $x \in \mathrm{Fix}(f)$, then $\gamma_x$ is a closed loop. 
A point $x\in \mathrm{Fix}(f)$ is called {\it{contractible}} for the isotopy $I$ 
if the loop $\gamma_x$ is null-homotopic; 
otherwise $x$ is called {\it{non-contractible}} for $I$.  
Likewise, if $x$ is a periodic point with minimal period $k$, then by concatenating the curves one arrives again at a closed loop $\Pi_{i={0}}^{k-1} \gamma_{f^{i}(x)}$. 
One defines, analogously to the fixed point case, contractible and non-contractible periodic points for the isotopy. 
The isotopy $I$ lifts in a unique way to an isotopy $\widetilde I= (\widetilde f_t)_{t\in[0,1]}$  in $\widetilde S$, such that $\widetilde f_0= \text{Id}$, and $\widetilde f= \widetilde f_1$ is a lift of $f$
commuting with all deck transformations in $\widetilde S$.
Note an $f$-periodic point $x \in S$ with minimal period $k$ 
is a contractible periodic point for the isotopy $I$ if, and only if, any lift 
$\widetilde x$ of $x$ is also
an $\widetilde f$-periodic point.
%In this case, one can also say that $x$ is a contractible or non-contractible periodic point for the lift $\widetilde f$.

Many recent progresses were made on relating the existence of non-contractible periodic points with boundedness of some orbits for the lifted dynamics. The main heuristics is usually the following idea.
In the absence of a very degenerate condition involving the fixed points set of the dynamics, 
either the diameter of an orbit of points with some sort of recurrent property in the lift is uniformly bounded, or one can find non-contractible periodic orbits with arbitrarily large minimal period. 
In~\cite{Tal_noncontractible} this was shown when the ``recurrent property in the lift'' 
meant to be periodic points in the lift, and in ~\cite{Forcing} the condition was weakened for points which are recurrent for the lift. In the present paper, our 
 first theorem improves both of these works, 
extending the result to the class of non-wandering points of the lifted dynamics. 
This is particularly helpful since the class of non-wandering points are much easier to work with than that of 
recurrent points. 

%\marginpar{changed the statement into a dichotomy statement.}
\begin{theorem}\label{non-wandering_Bounded_M} 
Let $S$ be a compact orientable surface, with the universal covering space $\widetilde S$.
Let $f$ be a homeomorphism of $S$ isotopic to the identity, with 
a lift $\widetilde f$ on $\widetilde S$ commuting with the deck transformations.
% (In the case $S=\Bbb T^2$, we also assume $f$ admit at least one fixed point). WE DONT NEED if $\widetilde f has a non-wandering point, then it has a fixed point.
Assume that the fixed point set of $f$ is contained in a topological open disk.
Then 
\begin{enumerate}
\item either $f$ does not admit non-contractible periodic orbits,
\item or there exists a constant $M>0$, 
 such that, for any $\widetilde f$-non-wandering point $\widetilde z$ and all integer $n$, 
 \begin{equation}
 \| \widetilde f^n (\widetilde z)- \widetilde z \| <M.
 \end{equation}
\end{enumerate}
\end{theorem}

We remark here that the hypothesis on the fixed point set cannot be removed, as there are examples where the dichotomy fails (see for example ~\cite{Koro_Tal_irrotational}).
However, failing to satisfy the hypotheses can only happen if there exists a homotopically non-trivial continuum of fixed points, which is a very degenerate condition.

In both ~\cite{Forcing} and ~\cite{Tal_noncontractible}, 
results on the existence of non-contractible periodic points were used to understand some torus homeomorphisms isotopic to the identity. 
Denote by $\text{Homeo}_0^+(\Bbb T^2)$ the set of homeomorphisms on $\Bbb T^2$
which is orientation-preserving and isotopic to the identity, and let 
$\widetilde {\text{Homeo}}_0^+(\Bbb T^2)$ be the set of lifts of homeomorphisms in 
$\text{Homeo}_0^+(\Bbb T^2)$ to the plane. Denote by $\text{Homeo}^+_{0,\text{nw}}(\Bbb T^2)$ the subset of $\text{Homeo}_0^+(\Bbb T^2)$ of homeomorphisms having no wandering points, and by $\widetilde {\text{Homeo}}^{+}_{0,\text{nw}}(\Bbb T^2)$ the set of their lifts to $\Bbb R^2$.  Next, $f\in \text{Homeo}_0^+(\Bbb T^2)$ is called 
{\emph{Hamiltonian}} if it preserve the Lebesgue measure and
it has a lift $\widetilde f$ to the universal covering space such that the rotation vector of $\widetilde f$ with respect to 
the Lebesgue measure is null 
(we refer to the next section for notation). In this case, we call $\widetilde f$ the Hamiltonian lift. 
Say a homeomorphism $f\in \text{Homeo}^+_0(\Bbb T^2)$ has a {\it uniformly bounded lift} $\widetilde f$ if there exists $M>0$  such that, for any $\widetilde z\in \Bbb R^2$ and all integer $n,\, \| \widetilde f^n (\widetilde z)- \widetilde z \| <M$. 
In the previous existing work on conditions ensuring the existence of non-contractible periodic points, 
it was usually fundamental to show that, 
if $f$ is a Hamiltonian homeomorphism 
with non-degenerate fixed point set, then its hamiltonian lift is a 
uniformly bounded lift (see Corollary I of ~\cite{Forcing}). 

In order to see a uniformly bounded lift, the hypothesis that the dynamics is Hamiltonian is crucial,
 since clearly one can consider an ergodic rigid translation in the torus.
 Moreover, if we assume that the dynamics has at least a periodic point there are examples of area-preserving homeomorphisms without any non-contractible periodic orbits, while at the same time, almost all points in the lift have unbounded orbits 
(see~\cite{Salvador_fareast} for an example, even in the smooth setting). 
However, these examples are very particular,
 as we will see more clearly below. 

For any $v\in \Bbb R^2\backslash \{(0,0)\}$, denote by $\text{pr}_v: w\mapsto \langle w,v\rangle/ \| v \|$
the projection of a vector $w$ to the oriented line passing through the origin in the direction $v$. We say $\widetilde f \in \widetilde {\text{Homeo}}_0^+(\Bbb T^2)$
has {\it{strong dynamical direction}} $v$, 
if it is not a uniformly bounded lift, 
and if there exists a constant $M>0$ such that, 
for any point $\widetilde z$ and any $n\ge 0$, 
it holds that 
$-M<\text{pr}_{v}(\tilde f^n(\tilde z)-\tilde z)$ 
and $-M<\text{pr}_{v^{\perp}}(\tilde f^n(\tilde z)-\tilde z)<M$ for any 
$v^{\perp}\in \Bbb R^2\backslash \{(0,0)\}$ which is perpendicular to $v$. 
If additionally $v$ has irrational slope, 
we say $\widetilde f$ has \emph{strong irrational dynamical direction $v$}.
For all known examples of area-preserving, non-Hamiltonian homeomorphisms with contractible periodic points but without non-contractible periodic points, the lifted map had strong irrational dynamical direction. Our next theorem shows that this is in fact the only possibility:

\begin{theorem}\label{rewritten_thm}
Let $f$ be an area-preserving homeomorphism with a lift $\widetilde f \in \widetilde {\text{Homeo}}_0^+(\Bbb T^2)$.
Assume that $f$ has contractible periodic points but 
no non-contractible periodic points, and that the fixed point set of $f$ is contained in a topological disk. Then either $f$ is Hamiltonian and $\widetilde f$ is a uniformly bounded lift, or $\widetilde f$  has a strong irrational dynamical direction.
\end{theorem}

Theorem~\ref{rewritten_thm} will be deduced as a consequence of several previous 
results on torus homeomorphisms in the isotopy class of the identity, 
as well as a result that is the fundamental new contribution in the second part of the present paper. The critical step needed to prove Theorem ~\ref{rewritten_thm} is to understand the dynamics of a particular class of homeomorphisms, a class whose rotation set is a line segment, such that one endpoint is a rational vector, and the other one is a irrational vector 
$(\alpha, \beta)$, where $\alpha/\beta$ is irrational. We postpone the introduction of the concept of rotation sets to the next section. Here we only comment this special 
 rotation set was studied in depth in the paper~\cite{Salvador_Liu}, where 
 discussions about the stable and unstable behaviour were given.
In particular, it was shown in~\cite{Salvador_Liu} that,
 if $\widetilde f$ is a lift of a non-wandering diffeomorphism, then under some broadly satisfied conditions $\widetilde f$ admits bounded deviation along the direction pointing to the rational endpoint of the segment (see Theorem 1.3 of~\cite{Salvador_Liu}). 
Here, using similar ideas as in the proof of Theorem~\ref{non-wandering_Bounded_M}, we are able to extend the bounded deviations result for all non-wandering homeomorphisms. 
\begin{theorem}\label{main_thm} 
Let $\widetilde f\in  
\widetilde{\text{Homeo}}^+_{0,\text{nw}}(\Bbb T^2)$, 
whose rotation 
set $\rho(\widetilde f)$ is a line segment from $(0,0)$ to some totally irrational 
vector $(\alpha,\beta)$. 
Then $\widetilde f$ has bounded deviation along the direction 
$-(\alpha,\beta)$.
\end{theorem}
This result, together with the main theorem (Theorem A) 
of~\cite{Guilherme_Thesis} will imply $\widetilde f$ has strong irrational dynamical direction.
We stress that, the same as in ~\cite{Guilherme_Thesis}, 
the proof of Theorems \ref{non-wandering_Bounded_M} and \ref{main_thm} rely on the newly developed forcing theory of transverse trajectories for surface homeomorphisms (see~\cite{Forcing}), and a crucial step in both is the use of the technical Proposition~\ref{unbounded_trajectory}, which may be  useful elsewhere. 
Let us refer to~\cite{Guilherme_Thesis} for more references, backgrounds and other related discussions. 

The rest of the paper will be organized as follows. 
In Section 2 we introduce notation and preliminaries. 
In Section 3, we prove Theorem~\ref{non-wandering_Bounded_M}.
In Section 4, we show Proposition~\ref{unbounded_trajectory}.
In section 5, we prove Theorem~\ref{main_thm} and then we prove Theorem~\ref{rewritten_thm}.

\section{Generalities and the Forcing Theory}\label{notation_section}

\subsection{Topological Dynamics on Surfaces and Torus}
Denote by $\Bbb S^1= \Bbb R/\Bbb Z$ the unit circle and by $\Bbb T^2= \Bbb S^1 \times \Bbb S^1$ the $2$-torus. Let $S$ be an oriented surface. We will denote by $\widetilde S$ the universal covering space of $S$, and by $\pi_{S}:\widetilde S\to S$ the covering projection. The set of \emph{Deck Transformations} of $S$ are the isometries of $\widetilde S$ that lift the identity map on $S$. 
A loop in $S$ is a continuous function $B:\Bbb S^1\to  S$. 
Given a loop $B$, let us call the natural lift of $B$ to the curve $\beta:\Bbb R \to S$ defined by 
$\beta(t)=B(\pi_{\Bbb S^1}(t))$, where $\pi_{\Bbb S^1}:\Bbb R\to \Bbb S^1$ is the canonical projection from the line on the circle. In particular $\beta$ is a $1$-periodic function.

If furthermore $S$ is a closed surface, a set $K\subset S$ is called \emph{inessential}
 if  it is contained in some topological disk $D\subset S$, otherwise it is called \emph{essential}. If the complement of $K$ is inessential, then $K$ is called \emph{fully essential}, in which case it must intersect any homotopically non-trivial loop in $S$. 
An open connected essential set $U$ is called \emph{annular} if it
is homeomorphic to an open annulus.%\marginpar{chamged definition, was not equivalent to the one in strictly toral}

For any closed surface $S$, denote by $\text{Homeo}^+_0(S)$ 
(respectively, $\text{Homeo}^+_{0,\text{nw}}(S)$) the set of 
orientation-preserving homeomorphisms (respectively, non-wandering orientation-preserving homeomorphisms) on
$S$ which are isotopic to the identity. 
Also denote by
$\widetilde{\text{Homeo}}^+_0(S)$ 
(respectively, $\widetilde{\text{Homeo}}^+_{0,\text{nw}}(S)$)
the set of homeomorphisms on $\Bbb R^2$ which are 
lifts of homeomorphisms in 
$\text{Homeo}^+_0(S)$ (respectively, $\text{Homeo}^+_{0,\text{nw}}(S)$). 
For any $f\in \text{Homeo}^+_0(S)$, 
a point $x \in S$ is called \emph{$f$-essential} if,
for any open set $U$ containing $x$, 
the set $W=\bigcup_{n\in \Bbb Z}f^n(U)$ is an essential set.
Otherwise, $x$ is called \emph{$f$-inessential}.
Denote by $\text{Ess}(f)$ the set of $f$-essential points, which is 
$f$-invariant and closed. Write $\text{Ine}(f)= S \backslash \text{Ess}(f)$, which is $f$-invariant and open. We refer to \cite{Strictly_Toral} for more information on these notions.

\subsection{Oriented Singular 
Foliation and Transverse Paths}
We call $\mathcal F$ an
\textit{oriented singular foliation}
 of $S$, if there exists a closed set $\text{Sing}(\mathcal F)\subset S$, 
called the \textit{singular set}, and if $\mathcal F$ is a partition of 
$\text{Dom}(\mathcal F): =S \backslash \text{Sing}(\mathcal F)$
into immersed real lines or circles, called \textit{leaves}, with a continuous choice of orientation. 
We also lift $\mathcal F$
to an oriented singular foliation $\widetilde{\mathcal F}$ on $\widetilde S$, 
where the singular set is $\text{Sing}(\widetilde{\mathcal F})=\pi_{S}^{-1}(\text{Sing}(\mathcal F))$, 
and therefore $\text{Dom}(\widetilde {\mathcal F})= \pi_{S}^{-1}(\text{Dom}(\mathcal F))$.
Given any $z\in S \backslash \text{Sing}(\mathcal F)$, there exists a local chart 
$h:W\to \Bbb R^2$, where $W$ is a neighbourhood of $z$, such that $h(\mathcal F\mid_W)$ is the foliation 
into vertical lines, oriented downward. In this case, 
we call $h$ a local \textit{trivialization chart around $z$}. 

A path $\gamma: I\to \text{Dom}(\mathcal F)$ is 
called \textit{transverse to $\mathcal F$} (or, $\gamma$ is an \textit{$\mathcal F$-transverse path}) 
if for any 
$t_0$ in the interior of $I$, 
there exist $\delta>0$ and a local trivialization chart $h$ around the point 
$\gamma(t_0)$, such that the function
\begin{equation}
t \mapsto \text{pr}_1\circ h\circ \gamma(t),
\end{equation}
is strictly increasing for $t\in (t_0-\delta,t_0+\delta)$, where $\text{pr}_1:\Bbb R^2 \to \Bbb R$ is the projection in the first canonical coordinate. 
A loop $B:\Bbb S^1 \to \text{Dom}(\mathcal F)$ is said to be $\mathcal{F}$-transverse if its natural lift is $\mathcal{F}$-transverse.

%Two $\mathcal F$-transverse paths  $\gamma_1,\gamma_2: \Bbb R\to \Bbb R^2$ 
%are called \textit{$\mathcal F$-equivalent} if for any $t\in \Bbb R$, 
%$\gamma_1(t)$ and $\gamma_2(t)$ lie in the same leaf of $\mathcal F$.

\subsection{Maximal Isotopy and Transverse Foliation}
Let $f: S \to S$ be a homeomorphism isotopic to the identity. 
Given an isotopy $I=\{f_t\}_{t\in[0,1]}$ from the $f_0=\text{Id}$ to $f_1=f$,
denote by $\text{Fix}(I)$ the set of points which are fixed by $f_t$ 
for all $t\in [0,1]$.
Write $\text{Dom}(I) =  S \backslash \text{Fix}(I)$. 
For any point $z\in \text{Dom}(I)$, 
consider the isotopy path 
\begin{align}\label{notation_isotopy}
\gamma_z = I^{[0,1]}(z): 
            [0,1] &     \to \text{Dom}(I), \\
                  t &     \mapsto f_t(z). \nonumber
 \end{align} 
Then we can extend the isotopy path in the following way. For any 
$t\in \Bbb R$, write $t=n_0+r$, where $n_0$ is an integer and $r\in [0,1)$. 
Then define 
\begin{equation}
f_t(z) =  f_r \circ f^{n_0}(z).
\end{equation}
By the \textit{whole isotopy trajectory} we mean the following.  
\begin{align}
I^{\Bbb R}(z): \Bbb R &\to \text{Dom}(I),\\ 
t &\mapsto f_t(z).  
\end{align}
In particular, $I^{\Bbb R}(z)(0)=f_0(z)=z.$ By abusing notation, sometimes we also 
call the image of a path by the same name. 
For example, denote by $I^{[a,b]}(z) =  I^{\Bbb R}(z)([a,b])$ the image of a finite isotopy trajectory. We also call \textit{whole isotopy trajectory} 
for $I^{\Bbb R}(z)(\Bbb R)$.

%\marginpar{\textcolor{blue}{ Before defining sucha  generl lemma with so many things, it better to use some "structuring concepts. We can define what is a transverse foliation to a maximal isotopy before, and the just say there exists a transverse foliation. Properties (1) and (2) from the lemma below will then be just part of the definition, and simplify the readers life}}

Given isotopies $I'=\{f'_t\}_{t\in [0,1]}$ and $I=\{f_t\}_{t\in [0,1]}$, we say that $I'\preceq I$ if $\mathrm{Fix(I')}\subset\mathrm{Fix}(I)$. An isotopy is said to be a \emph{maximal isotopy} if it is a maximum element for this partial order. It was shown in ~\cite{BCL_Fixed} that given any isotopy $I'$ one could always find a maximal isotopy larger than it (see also~\cite{Jaulent} for a previous existence result which was already enough for most applications). A crucial property of a maximal isotopy $I=(f_t)_{t\in [0,1]}$ is that, for every $z\in \mathrm{Fix}(f_1)\setminus\mathrm{Fix}(I)$, the closed curve $\gamma_z$ is not contractible in $\text{Dom}(I)$.

A path is called \emph{transverse to $\mathcal F$} (or \emph{$\mathcal F$-transverse}) if locally it passes through each leaf from its right to its left. 
An oriented singular foliation $\mathcal F$ on $S$ is called a \emph{transverse foliation for the maximal isotopy $I$}, if the singular set $\text{Sing}(\mathcal F)$ coincides with the fixed point set $\text{Fix}(I)$, 
and if  for any $z\in \text{Dom}(I)$, the isotopy path $\gamma_z$ is homotopic in $\text{Dom}(I)$, relative to its endpoints,
to a path which is transverse to the foliation $\mathcal F$. 
Write $\text{Dom}(\mathcal F)= S\backslash \text{Sing}(\mathcal F)$ 
and note that $\text{Dom}(\mathcal F)=\text{Dom}(I)$. $\mathcal{F}$ is also called a \emph{Brouwer-Le Calvez} foliation.

For any $z\in \text{Dom}(I)$, 
denote by $I_{\mathcal F}(z)$ the set of $\mathcal F$-transverse paths that are homotopic in $\text{Dom}(I)$, relative to its endpoints, to $\gamma_z$. 
Also, for $n\geq 1$, denote by $I_{\mathcal F}^n(z):=\prod_{k=0}^{n-1}I_{\mathcal F}(f^kz)$ the set of transverse paths obtained by concatenating the corresponding transverse paths. 
Similarly, denote $I_{\mathcal F}^{\Bbb R}(z)$ the set of infinite transverse path $\prod_{k=-\infty}^{\infty}I_{\mathcal F}(f^kz)$. By abusing notation, when the context is clear, 
we sometimes use $I_{\mathcal F}^n(z)$ or $I_{\mathcal F}^{\Bbb R}(z)$ to denote one specific transverse path contained in the corresponding sets of paths.

The following important theorem, stated here as a lemma, was a breakthrough result and it 
provides a completely new tool for understanding surface dynamics in general. 
\begin{lemma}[see~\cite{LC2005}]\label{transverse_foliation_thm}
Let $f\in \text{Homeo}^+_0(S)$ and let $I=(f_t)_{t\in[0,1]}$ be a maximal isotopy such that $f_0=Id$ and $f_1=f$. 
Then there exists a transverse foliation $\mathcal F$ for $I$. 
\end{lemma}

Note that $f, I$ and $\mathcal F$ are naturally lifted to $\widetilde f$, $\widetilde I$ and $\widetilde {\mathcal F}$ respectively
in the covering space $\widetilde S$ of $S$. These objects a compatible int he following sense. 
$\widetilde I$ is a maximal isotopy from $Id$ to $\widetilde f$, and $\widetilde F$ is a transverse foliation to $\widetilde I$.
We find that $\text{Dom}(\widetilde{\mathcal F})=\text{Dom}(\widetilde I)$ is the complement of 
the corresponding singular set for both the isotopy $\widetilde I$ and the foliation 
$\widetilde{\mathcal F}$.
Furthermore, as the restriction of $\widetilde I$ to $\text{Dom}(\widetilde I)$ is also isotopic to the identity, one can lift all these objects to the universal covering of this latter surface, namely $\text{Dom}(\mathcal F)^{\text{uni}}$, which is homeomorphic to a union of disjoint copies of $\Bbb R^2$. The lifted objects will be denoted $f^{\text{uni}}$, $I^{\text{uni}}$ and $\mathcal {F}^{\text{uni}}$. The restriction of  $f^{\text{uni}}$ to each connected component of $\text{Dom}(\mathcal F)^{\text{uni}}$ is a Brouwer homeomorphism, 
which is, by definition, an orientation-preserving homeomorphism on $\Bbb R^2$
without fixed points. Finally, $\mathcal {F}^{\text{uni}}$ is a non-singular foliation, and each leaf 
$\ell^{\text{uni}}$ of $\mathcal {F}^{\text{uni}}$ is an $f^{\text{uni}}$-Brouwer line, that is, $f^{\text{uni}}(\ell^{\text{uni}})$ is contained in the left of $\ell^{\text{uni}}$, and $(f^{\text{uni}})^{-1}(\ell^{\text{uni}})$ is contained in the right of $\ell^{\text{uni}}$.

We need the following definitions from~\cite{Forcing}:
\begin{definition} Let $\mathcal F$ be an oriented foliation on $\Bbb R^2$, either singular or non-singular. 
Let $\gamma_0:[a,b]\to S \backslash \text{Sing}(\mathcal F), \gamma_1:[a',b']\to S \backslash \text{Sing}(\mathcal F)$ be two $\mathcal F$-transverse paths. 
We say $\gamma_0$ and $\gamma_1$ are 
\emph{$\mathcal F$-equivalent} %\textcolor{black}{I think I prefer the original one},\footnote{We note that in the original paper~\cite{Forcing}, these two trajectories are said to be 
%$\mathcal F$-equivalent.}
 if there is a homotopy $H: [a,b]\times [0,1]\to S\backslash \text{Sing}(S)$ and an increasing homeomorphism $\phi:[a,b]\to [a',b']$,
such that for all 
$c\in[a,b], 
H(c,0)=\gamma_0(c), 
H(c,1)=\gamma_1(\phi(c))$, 
and the image of $H(c,t) \big |_{t\in [0,1]}$ is contained in a single leaf of $\mathcal F$. 
\end{definition}
%\marginpar{NOT EQUIVALENT< NEED TO CHECK ORIGINAL DEFINITION.}
\begin{definition} 
Let $J$ be an interval (possibly unbounded) in $\Bbb R$. 
We say an $\mathcal F$-transverse path $\gamma: J \to \text{Dom}(I)$ 
is admissible if for any compact subinterval $J_1\subset J$, there exists $z\in \text{Dom}(I)$ such that the restriction of $\gamma$ to $J_1$ is $\mathcal F$-equivalent to a  subpath of $I_{\mathcal{F}}^{\Bbb R}(z)$. Furthermore, if there exists $z\in \text{Dom}(I)$ and a positive integer $n$ such that $\gamma$ is 
$\mathcal F$-equivalent to an element in 
$I_{\mathcal F}^n(z)$, then $\gamma$ is said to be \textit{admissible of order $n$}. If $\gamma$ is $\mathcal F$-equivalent with a subpath of such an element, then it is said to be admissible of order $\leq n$.
\end{definition}

The following continuity property for admissible $\mathcal F$-transverse paths will be useful. 
\begin{lemma}[Lemma 17 of \cite{Forcing}]\label{lemma_of_continuity} 
Let $I$ be a maximal isotopy and $\mathcal F$ is a transverse foliation 
for $I$. For any $z\in \text{Dom}(I)$ and any $n\geq 1$, there exists a neighbourhood $W$ of $z$ such that, for every $z', z''\in W$, the path $I^n_{\mathcal F}(z')$ is $\mathcal F$-equivalent to a subpath of $I^{n+2}_{\mathcal F}(f^{-1}(z''))$. 
\end{lemma}

The following result focuses on the torus case. 
\begin{lemma}[Lemma 4.8 of~\cite{Strictly_Toral}, essentially dating back to Section 10 of~\cite{LC2005}]\label{gradient-like} 
Suppose we are working with a maximal isotopy $I$ and the transverse foliation $\mathcal F$, which are lifted to 
$\widetilde I$ and $\widetilde {\mathcal F}$ in $\Bbb R^2$. 
Assume an $\widetilde {\mathcal F}$-transverse loop $\Sigma$ is written as a concatenation of finitely many transverse loops $\{\beta_i\}_{i=1}^p$, with a same base point,
which induce the homology directions $\{\beta_i^\ast\}_{i=1}^p$. 
Assume these directions $\{\beta_i^\ast\}_{i=1}^p$ 
can represent any element in the first homology group $H^1(\Bbb T^2,\Bbb Z)$,
 with all coefficients positive integers,
and assume $\sum_{i=1}^p \beta_i^\ast=0$. 
Then every leaf of $\widetilde {\mathcal F}$ is uniformly bounded. 
\end{lemma}

\subsection{\textcolor{black}{Forcing results}}
Consider an oriented non-singular foliation $\mathcal F$ of $\Bbb R^2$, 
and denote by $\phi_0,\phi_1$ and $\phi$ 
three distinct leaves of $\mathcal F$, parameterized according to their orientation, and satisfying that none of which separate the other two. We say \emph{$\phi_1$ is above $\phi_0$ relative to $\phi$ (and $\phi_0$ is below $\phi_1$ relative to $\phi$)}, if for some $t_0<t_1$, writing $z_0=\phi(t_0)$ and $z_1=\phi(t_1)$, 
there exist disjoint paths $\lambda_0$ and $\lambda_1$, 
such that for $i=0,1$, $\lambda_i$ joints $z_i$ to  $x_i\in \phi_i$, 
and moreover $\lambda_0$ and $\lambda_1$ do not intersect any of the three leaves except at the endpoints.  %\marginpar{Is this statement for "above" and "below" equivalent?}

For any $z\in \text{Dom}(\mathcal F)$, denote by $\phi_z$ the leaf in $\mathcal F$ which contains $z$.
Let 
$\gamma_0 : [a_0,b_0]\to \Bbb R^2$ and $\gamma_1:[a_1,b_1]\to \Bbb R^2$ 
be two $\mathcal F$-transverse paths. Suppose also $\gamma_0(t_0)=\gamma_1(t_1)$ belongs to some leaf $\phi\in \mathcal F$, for certain times $t_0\in (a_0,b_0)$ and $t_1\in(a_1,b_1)$. We say 
$\gamma_0\big |_{[a_0,b_0]}$ and $\gamma_1\mid_{[a_1,b_1]}$ intersect \emph{$\mathcal F$-transversely}
(at the leaf $\phi$), if 
$\phi_{\gamma_0(a_0)}$ is below $\phi_{\gamma_1(a_1)}$ relative to $\phi$, 
and 
$\phi_{\gamma_0}(b_0)$ is above $\phi_{\gamma_1(b_1)}$ relative to $\phi$.

More generally, suppose $\mathcal F$ is an oriented singular foliation on a surface $S$. Let again 
$\gamma_0: [a_0,b_0]\to \text{Dom}(\mathcal F),
\gamma_1:[a_1,b_1]\to \text{Dom}(\mathcal F)$ be $\mathcal{F}$-transverse paths, and assume 
there exists a leaf $\phi$ containing both $\gamma_0(t_0)$ and $\gamma_1(t_1)$ for some times $t_0\in(a_0,b_0)$ and $t_1\in (a_1,b_1)$. We say $\gamma_0$ and $\gamma_1$ intersect $\mathcal F$-transversely (at the leaf $\phi$) if, $\gamma_0$ and $\gamma_1$ can be lifted to 
the paths $\gamma_0^{\text{uni}}$, $\gamma_1^{\text{uni}}$ on the universal covering space $\text{Dom}(\mathcal F)^{\text{uni}}$ (see the paragraph after Lemma~\ref{transverse_foliation_thm}), such that $\gamma_0^{\text{uni}}(t_0)$ and  $\gamma_1^{\text{uni}}(t_1)$ belong to the same leaf $\phi^{\text{uni}}$, and such that, 
$\gamma_0^{\text{uni}}$ and  $\gamma_1^{\text{uni}}$ intersect $\mathcal F^{\text{uni}}$-transversely at the leaf $\phi^{\text{uni}}$. Note that it is possible for a path to have a $\mathcal F$-transverse intersection with itself by considering two distinct lifts of the path to $\text{Dom}(\mathcal F)^{\text{uni}}$. In this case we say that the path has a $\mathcal{F}$-transverse self-intersection.

%\marginpar{\textcolor{blue}{WE need the lemma on shortening paths by having no transverse intersection.}}
The central result of \cite{Forcing}, 
often referred to as the \emph{Forcing Proposition}, is stated in the following lemma.
\begin{lemma}[Proposition 20 of~\cite{Forcing}]\label{Forcing_proposition} 
Let $f$ be a homeomorphism of a surface $S$ isotopic to the identity, $I$ a maximal isotopy for $f$, $\mathcal F$ a singular oriented foliation transverse to $I$. Let 
$\gamma_0:[a_0,b_0]\to \text{Dom}(\mathcal F)$ and $\gamma_1:[a_1,b_1]\to \text{Dom}(\mathcal F)$ be two $\mathcal F$-transverse paths and assume they intersect $\mathcal F$-transversely at $\phi= \phi_{\gamma_0(t_0)}$ and that $\gamma_0(t_0)=\gamma_1(t_1)$. If $n_0, n_1$ are positive integers such that $\gamma_0$ is admissible of order $n_0$ and 
$\gamma_1$ is admissible of order $n_1$, then the concatenations $\gamma_0\mid_{[a_0,t_0]} \gamma_1\mid_{[t_1,b_1]}$ and $\gamma_1 \mid_{[a_1,t_1]}\gamma_0\mid_{[t_0,b_0]}$ are both admissible of order $n_0+n_1$.
\end{lemma}

We will also need the following result, whose proof follows exactly that of Corollary 24 of ~\cite{Forcing}:
\begin{lemma}\label{shortening_paths}
 Let $\gamma:[a,b]\to \text{Dom}(\mathcal F)$ be an $n$-admissible $\mathcal{F}$-transverse path. Given $a\le a'<b'\le b$, there exists an $n$-admissible transverse path $\gamma  
'$ that is the same as $\gamma$ in $[a,a']$ and in $[b',b]$ and such that $\gamma'\mid_{[a',b']}$ has no $\mathcal{F}$-transverse self-intersection, and the image of $\gamma'\mid_{[a',b']}$ is contained in that of  $\gamma\mid_{[a',b']}$.
\end{lemma}

Consider a singular transverse foliation $\widetilde {\mathcal F}$ on $\Bbb R^2$, 
and fix an $\widetilde {\mathcal F}$-transverse  \textcolor{black}{simple} loop $\widetilde B$, whose natural lift is denoted by $\widetilde \beta$.
The set of leaves that are intersected by $\widetilde\beta$ is a topological open annulus, which we denote by $U_{\widetilde \beta}$.
Assume $\widetilde \gamma: J \to \text{Dom}(\widetilde {\mathcal F})$ is an $\widetilde {\mathcal F}$-transverse path. 
Following~\cite{Calvez_Tal_topological}, let us make the following useful definitions.

\begin{definition}[drawing and crossing loops for transverse paths] 
We say $\widetilde \gamma$ draws the loop $\widetilde B$ whose natural lift is $\widetilde \beta$, if there exist times $a < b$ such that 
$\widetilde \gamma\mid_{[a,b]}$ is $\widetilde {\mathcal F}$-equivalent to $\widetilde \beta\mid_{[t,t+1]}$ for some $t$. We say $\widetilde \gamma$ crosses $\widetilde B$, if there are times $a,b\in J$ such that 
$\widetilde \gamma(a)$ and $\widetilde \gamma(b)$ belong to two different connected components of the complement of $U_{\widetilde \beta}$, \textcolor{black}{at least one of which is bounded}. When $\widetilde \gamma: J\to \text{Dom}(\widetilde {\mathcal F})$ draws $\widetilde B$ (respectively, crosses $\widetilde B$), a connected component $I$ of the set $\{t\in J \mid \widetilde \gamma(t) \in U_{\widetilde \beta}\}$ is called a drawing component (respectively, crossing component) if $\widetilde \gamma \mid_I$ draws $\widetilde B$ (respectively, if $\widetilde \gamma \mid_{I}$ crosses $\widetilde B$).  
\end{definition}
The following lemma is a reformulation of Proposition 24 of~\cite{Calvez_Tal_topological}. 
\begin{lemma}\label{drawing_crossing_transverse}
Assume $S$ has genus $0$. Denote by $\gamma_0, \gamma_1$ two ${\mathcal F}$-transverse paths, and $ \beta$ an $ {\mathcal F}$-transverse loop.
Suppose a finite interval $J_0=[a_0,b_0]$ is such that $(a_0, b_0)$ is a drawing component of $ \gamma_0$,
and a finite interval $J_1=[a_1,b_1]$ is either a drawing component or a crossing component (eventually the same as $J_0$).
Assume also $\gamma_0(b_0)$ and $\gamma_1(a_1)$ belong to the same connected component of the complement of $U_{\beta}$.
Then $\gamma_0 \mid_{J_0}$ and $\gamma_1 \mid_{J_1}$
have an ${\mathcal F}$-transverse intersection.
\end{lemma}
\begin{proof} The conclusion follows from the proof of Proposition 24 of~\cite{Calvez_Tal_topological}. 
\end{proof}

The following result inspired our main theorem, which was already described in the introduction. We say an $\mathcal F$-transverse path $\gamma:\Bbb R\to \text{Dom}(\mathcal F)$ is $\mathcal F$-recurrent if for any compact interval $J\subset \Bbb R$ and any 
$t\in \Bbb R$, there exists some interval $J'\subset [t,\infty)$, such that $\gamma\big |_{J'}$
and $\gamma\mid_{J}$ are $\mathcal F$-equivalent.
If it also satisfies that for any $t$ there is $J''\in(-\infty, t]$ such that $\gamma\big |_{J''}$ and $\gamma\mid_{J}$ are $\mathcal F$-equivalent, then we say $\gamma$ is $\mathcal F$-bi-recurrent. 

\begin{lemma}\label{three_transformations}[Proposition 43 of \cite{Forcing}]
Let $S$ be a surface, $f$ a homeomorphism of $S$ isotopic to the identity, $I$ a maximal isotopy for $f$ and $\mathcal{F}$ a singular transverse foliation for $I$. Let  $\widetilde S$ be the universal covering space of $S$, and let $\widetilde f, \widetilde I, \widetilde{\mathcal  F}$ be, respectively,  lifts of  $f, I, \mathcal F$ to $\widetilde S$.
Assume an $\widetilde{\mathcal F}$-transverse path 
$\widetilde \gamma: \Bbb R\to \text{Dom}(\widetilde{\mathcal F})$ 
is admissible and 
is $\widetilde {\mathcal F}$-bi-recurrent, 
and there are three distinct deck transformations $T_1,T_2,T_3$ and certain leaf $\widetilde \phi$ of $\widetilde {\mathcal F}$, such that $\widetilde \gamma$ intersects all three translates $T_i(\widetilde \phi)$, $i=1,2,3$. Then there exists non-contractible periodic orbits for $f$. 
\end{lemma}

The previous lemma is dependent of the next one, which we also use in this work:
\begin{lemma}[Proposition 26 of~\cite{Forcing}]\label{translate_intersection}
Suppose an admissible path $\widetilde \gamma$ and a non-trivial deck translate of $\widetilde \gamma$ intersect $\widetilde {\mathcal F}$-transversely, then $f$ admits non-contractible periodic orbits.  
\end{lemma}

\subsection{Torus Rotation Set Theory and Bounded Deviations}
This subsection is specific for dynamics on the two torus $\Bbb T^2$. 
Given $\widetilde f\in \widetilde{\text{Homeo}}^+_0(\Bbb T^2)$ which is a lift of $f\in \text{Homeo}^+_0(\Bbb T^2)$, the Misiurewicz-Ziemann rotation set (as introduced in~\cite{MZ}) is defined as follows.
\begin{equation}\label{pp_rotation_vector}
\rho(\widetilde f):=\{\rho \big | \rho = \lim_{i\to \infty}\frac{1}{n_i}\big( \widetilde f^{n_i}(\widetilde z_i)-\widetilde z_i \big), \text{for some } \widetilde z_i\in \Bbb R^2 \text{ and some } n_i\to +\infty\}.
\end{equation}
One can also define a rotation vector for a fixed point $z\in \Bbb T^2$, 
\begin{equation}
\rho(\widetilde f,z):=\lim_{n\to \infty}\frac 1n \big( \widetilde f^n(\widetilde z)-\widetilde z\big),\, \widetilde z \in \pi^{-1}_{\Bbb T^2}(z),
\end{equation}
which is only well-defined when the limit exists. Note that in this definition, the limit does not depend on the choice of the lift $\widetilde z$ of $z$ when the limit exists. 
Moreover, we can denote $\mathcal M_f(\Bbb T^2)$ the set of $f$-invariant Borel 
probability measures on $\Bbb T^2$.
Also denote by $\mathcal M_{f,E}(\Bbb T^2) \subseteq \mathcal M_f(\Bbb T^2)$ the subset of ergodic measures. For the lift $\widetilde f$, consider the displacement function $\Delta_{\widetilde f}: \Bbb T^2 \to \Bbb R^2$, such that 
$\Delta_{\widetilde f}(z)=\widetilde f(\widetilde z)-\widetilde z$. Note again here the choice of the lift $\widetilde z$ does not affect this value, since $\widetilde f$ commutes with all deck transformations (i.e., integer translations). For any $\mu\in\mathcal M_f(\Bbb T^2)$, define $\rho_{\mu}(\widetilde f):=\int_{\Bbb T^2}\Delta_{\widetilde f}(z)d\mu(z)$. Then define $\rho_{\text{mes}}(\widetilde f):=\{\rho_{\mu}(\widetilde f) \mid \mu\in \mathcal M_f(\Bbb T^2)\}$. A rotation vector  $v\in \Bbb R^2$ is said to be realized by a measure $\mu$ if $\rho_{\mu}(\widetilde f)=v$. One has: 
\begin{lemma}[See~\cite{MZ}, especially Section 3]\label{MZ_rotation_properties} 
For any $\widetilde f\in \widetilde{\text{Homeo}}^+_0(\Bbb T^2)$, 
$\rho(\widetilde f)$ and $\rho_{\text{mes}}(\widetilde f)$ coincide, 
and it is always a compact and convex subset of $\Bbb R^2$. 
Moreover, the extremal points of $\rho(\widetilde f)$ are always realizable by some $\mu\in \mathcal M_f(\Bbb T^2)$.
\end{lemma}
A homeomorphism $f\in \text{Homeo}^+_0(\Bbb T^2)$ is called irrotational if $\rho(\widetilde f)=\{(0,0)\}$ for some lift $\widetilde f$ of $f$, which we call 
the the irrotational lift. 
$f\in \text{Homeo}^+_0(\Bbb T^2)$ is called \emph{Hamiltonian}
if it preserves a Borel probability measure $\mu$ with full support and no atoms, and $\rho_\mu(\widetilde f)=\{(0,0)\}$, for some lift $\widetilde f$ of $f$, which we call the \emph{Hamiltonian lift}. 
Now we can restate the the motivating result from \cite{Forcing}.
\begin{lemma}[Corollary I of~\cite{Forcing}]\label{hamiltonianlemma} Let $f\in \text{Homeo}^+_0(\Bbb T^2)$ be Hamiltonian with Hamiltonian lift $\widetilde f$. Assume its fixed point set is contained in a topological disk. Then either
$\widetilde f$ is a uniformly bounded lift, or $f$ admits non-contractible periodic orbits. 
\end{lemma}

Several realization results are recorded in the following lemma.
\begin{lemma}[Theorem 3.2 of~\cite{franks1989realizing}]\label{franks1989realizing}
Given $\widetilde f\in \widetilde{\text{Homeo}}^+_0(\Bbb T^2)$ and suppose $\rho(\widetilde f)$ has non-empty interior. Then every rational vector $v$ in the interior of $\rho(\widetilde f)$ is realizable by a periodic orbit, that is, if $v=(\frac {p_1}{q},\frac{p_2}{q})$ is written in irreducible form, there exists a periodic point $x$ with period $q$ such that $\widetilde f^q(\widetilde x)=\widetilde x+(p_1,p_2)$. In particular, if $v$ is not null, $x$ lies in a non-contractible periodic orbit. 
\end{lemma}

\begin{lemma}[Main result of~\cite{Franks2}\label{lemmafranks2}]Let $\widetilde f \in \widetilde{\text{Homeo}}^+_{0, \textrm{nw}}(\Bbb T^2)$ and $\rho(\widetilde f)$ is a line segment. If $v\in \rho(\widetilde f)$
is a rational vector, then it is realizable by a periodic orbit. In particular, if $\rho(\widetilde f)$ contains at least two distinct rational vectors, then $f$ admits non-contractible periodic orbits. 
\end{lemma}

Another relevant result, also from \cite{Forcing}, resolved one case of the famous Franks-Misiurevicz conjecture. 
\begin{lemma}[Theorem C of~\cite{Forcing}]\label{ThmC_Forcing} Let $\widetilde f\in \widetilde{\text{Homeo}}^+_0(\Bbb T^2)$. If 
the boundary of $\rho(\widetilde f)$ includes a line segment with irrational slope and contains exactly one rational point, then the rational point must be of the endpoints of this segment. 
\end{lemma}

Recall that for any $v\in \Bbb R^2\backslash \{(0,0)\}$,  
$\text{pr}_v: w\mapsto \langle w,v\rangle/ \| v \|$ denotes the projection of 
a vector $w$ into the line passing through the origin in the direction $v$. 
We say $\widetilde f$ has \emph{bounded deviation 
(from its rotation set $\rho(\widetilde f)$)} 
along a direction $v\in \Bbb R^2 \backslash \{(0,0)\}$, 
if there exists a uniform constant $C>0$, 
such that for any $n\geq 1$, and for any lift $\widetilde x$ of $x$,
\begin{equation}
\text{pr}_v (\widetilde f^n(\widetilde x)-\widetilde x-n\rho(\widetilde f))<  C.
\end{equation}
We will be mainly interested in the case when $\rho(\widetilde f)$ is a line segment containing $(0,0)$ and some non-null vector $(\alpha,\beta)$. In this case we say that the rotation set has direction $(\alpha, \beta)$.

%The next lemma, due to Atkinson, is very useful on estimating the deviations from the rotation set.
%\begin{lemma}[See~\cite{Atkinson}. See also Proposition 47 of~\cite{Forcing} in the same context.]\label{atkinson_lemma}
%Suppose $f\in \text{Homeo}_0^+(\Bbb T^2)$ preserves an ergodic probability measure $\mu$.
%Let $\varphi:\Bbb T^2 \to \Bbb R$ be a continuous function, with $\int_{\Bbb T^2}\varphi(x) d\mu(x)=0$,
%which induces the real-valued cocycle $\varphi^{(n)}(x):=\sum_{k=0}^{n-1} \varphi\circ f^n(x)$. Then for 
%$\mu$-almost every $x\in \Bbb T^2$, there exists integer sequence $n_k$ tending to 
%$\infty$, so that $f^{n_k}(x)\to x$ and $\varphi^{(n_k)}(x)\to 0$.%{\marginpar{I REMOVED A SENTENCE}}% Moreover, we can also require each term of $\varphi^{(n_k)}(x)$ to be non-negative or non-positive. 
%\end{lemma}
%\marginpar{did we use atkinson somewhere?}

We next state several recent results on establishing bounded deviation under different contexts.
The first case we consider is when the rotation set has interior. 

\begin{lemma}[Theorem D of~\cite{Forcing}]\label{deviation_forcing}
Assume $\widetilde f \in \widetilde{\text{Homeo}}^+_0(\Bbb T^2)$ and 
$\rho(\widetilde f)$ has non-empty interior. Then $\widetilde f$ admits bounded deviation along every direction. 
\end{lemma}

The following two results focus on perpendicular directions to rotation sets which are line segments.

\begin{lemma}[Theorem A of~\cite{davalos2018annular}]\label{davalos2018annular}
Suppose $\widetilde f \in \widetilde{\text{Homeo}}^+_0(\Bbb T^2)$ 
and $\rho(\widetilde f)$ is a line segment with rational slope, containing $(0,0)$ and with direction $v$. 
Then $\widetilde f$ admits bounded deviation along both perpendicular directions $v^{\perp}$ and $-v^{\perp}$.
\end{lemma}

\begin{lemma}[Theorem A of~\cite{Guilherme_Thesis}]\label{Guilherme}
For some $\widetilde f \in \widetilde{\text{Homeo}}^+_0(\Bbb T^2)$, assume $\rho(\widetilde f)$ is a 
line segment from $(0,0)$ to some vector $v=(\alpha,\beta)$, where $\alpha/\beta$ is irrational, then $\widetilde f$ admits 
bounded deviation along the perpendicular directions $v^{\perp}$ and $-v^{\perp}$ .
\end{lemma}

The following is a direct combination of the above several lemmas. 
We state it for our convenience to apply later.
\begin{lemma}\label{Guelman}
Assume $\rho(\widetilde f)$ contains $(0,0)$, and 
one of the following sets of conditions holds. 
\begin{enumerate}
\item either $\rho(\widetilde f)$ is a non-trivial line segment, and $v\in \Bbb S^1$ is a direction perpendicular to this segment. 
\item or $\rho(\widetilde f)$ has non-empty interior, and $v \in \Bbb S^1$ 
is a direction such that $\langle r,v\rangle\leq 0$ for all $r\in \rho(\widetilde f)$.
\end{enumerate}
Then, $\widetilde f$ admits bounded deviation along $v$ .
\end{lemma}

%\begin{proof} Assume condition (1). If $\rho(\widetilde f)$ is a non-trivial segment with irrational slope,  then we can apply Lemma~\ref{ThmC_Forcing}. It follows $(0,0)$ must be one of the endpoint of the segment. Then by Lemma~\ref{Guilherme}, $\widetilde f$ admits bounded deviation along the perpendicular direction of $\rho(\widetilde f)$.If $\rho(\widetilde f)$ has rational slope, it follows from Lemma~\ref{davalos2018annular} that $\widetilde f$ admits bounded deviation along the perpendicular direction of $\rho(\widetilde f)$.

%Now we assume condition (2).The conclusion follows from Lemma~\ref{deviation_forcing}. \end{proof}

%\begin{Remark}\label{Pablo_annular} In the case when $\rho(\widetilde f)$ is a line segment with rational slope and contains $(0,0)$, 
%Theorem A in \cite{davalos2018annular} asserts stronger conclusion that 
%the dynamics $f$ is in fact annular. 
%\end{Remark}

%\marginpar{\textcolor{black}{We also need to say soemthings like - non-wandering homeos whose rotation set has two points in $Q^2$ must have  non-contractible periodic points. If interior of rotation set is not empty, has contractible periodic points. the case that Patrice and I showed could not exists. So if an area-preserving homeomorphism has fixed points but no noncontractible periodic points, the ints rotation set is just null, meaning its hamiltonian and we can apply Corollary I from \cite{Forcing}, or it has the shaope we study here}}

\section{Non-wandering points and Non-contractible Periodic Orbits}
\textcolor{black}{Assume $S$ is a closed connected surface of genus $g\ge 1$, $f\in \text{Homeo}^+_0(S)$, $I$ is a maximal isotopy for $f$ and $\mathcal F$ is a foliation transverse to $I$. Further assume that $\widetilde {\mathcal F}$ and $\widetilde I$ are the lifts of 
$\mathcal F$ and $I$ to $\widetilde S$, the universal covering space of $S$, which is homeomorphic to $\Bbb R^2$.}
The main result of this section is the following proposition, which states that,
if an admissible path intersects twice five distinct deck images of one leaf in a nice way, 
then $f$ admits non-contractible periodic orbits. 

\begin{prop}\label{new_prop}
%Let Given $\varepsilon>0$ and five moments $0=t_0<t_1<t_2<t_3<t_4<1$, let $\widetilde \beta:[-\varepsilon, 1+t_4+\varepsilon]$ be an admissible $\widetilde {\mathcal F}$-transversepath, satisfying the following properties. Assume for a leaf $\widetilde \phi_0$ of $\widetilde {\mathcal F}$, and for five pair-wise distinct deck transformations $\{u_i\}_{i=0}^4$, we have that $\widetilde \beta(t_i)=\widetilde \beta(1+t_i)\in u_i(\widetilde \phi_0)$ for $i=0,1,2,3,4$.
Let $0=t_0<t_1<t_2<t_3<t_4<1$, and assume there exists 
$\widetilde \beta:[0, 1+t_4] \to 
\text{Dom}(\widetilde {\mathcal F})$ which is an admissible 
$\widetilde {\mathcal F}$-transverse path, satisfying 
the following properties:
\begin{itemize}
\item There exists a  leaf $\widetilde \phi_0$ of $\widetilde {\mathcal F}$, and five pairwise distinct 
deck transformations $u_i, 0\le i\le 4$, such that $\widetilde \beta(t_i)=\widetilde \beta(1+t_i)\in u_i(\widetilde \phi_0)$ for $i=0,1,2,3,4$.
\item \textcolor{black}{If $t\in(0, t_1)$} % is such that $t\notin \{ t_1, t_2, t_3\}$,
 then both $\widetilde \beta(t)$ and $\widetilde \beta(1+t)$ do not belong to $u(\widetilde \phi_0)$ for any deck transformation $u$.
\end{itemize}
Then $f$ admits non-contractible periodic orbits. 
\end{prop}

The proof of this proposition is postponed to subsection~\ref{five_proof}.

\subsection{Proof of Theorem~\ref{non-wandering_Bounded_M} using Proposition~\ref{new_prop}}\label{section_3.1}
Let us first show an immediate consequence of Proposition~\ref{new_prop}.

\begin{coro}\label{five_translates_we_will_use}
Let $\widetilde x_0\in \Omega(\widetilde f)$, and assume that 
\textcolor{black}{there are distinct deck transformations $u_i, \, 0\le i\le 4$,  and a  leaf $\widetilde\phi_0$ of $\widetilde {\mathcal F}$,
 such that the transverse trajectory 
 $\widetilde{I}_{\widetilde{\mathcal{F}}}^{\Bbb R}(\widetilde x_0)$ intersects $u_i(\widetilde\phi_0)$ for each $0\le i\le 4$. }
Then, $\widetilde f$ admits non-contractible periodic points. 
\end{coro}

\begin{proof}
Let us fix the leaf $\widetilde \phi_0 \in \widetilde {\mathcal F}$ 
and also the deck transformations $u_0,\cdots, u_4$. Up to replacing $\widetilde x_0$ with some $\widetilde f^{-k}(\widetilde x_0)$ if necessary, we can assume that 
the transverse trajectory 
$\widetilde I_{\widetilde {\mathcal F}}^T(\widetilde x_0)$ intersects each of 
$u_0( \widetilde \phi_0), u_1( \widetilde \phi_0),\cdots,u_4( \widetilde \phi_0)$ (in that order), for some large integer $T$. 
 By Lemma~\ref{lemma_of_continuity}, 
 we can find a small neighbourhood $\widetilde V$ of $\widetilde x_0$, with the following properties. 
 \begin{enumerate}
 \item for any $\widetilde y, \widetilde y'\in \widetilde V$, the transverse path $\widetilde I_{\widetilde {\mathcal F}}^T(\widetilde y)$ is $\widetilde {\mathcal F}$-equivalent to a subpath of the trajectory 
$\widetilde I_{\widetilde {\mathcal F}}^{T+2}(\widetilde f^{-1}(\widetilde y'))$.
\item for any point $\widetilde y\in \widetilde V$, the \textcolor{black}{transverse} 
path $\widetilde I_{\widetilde {\mathcal F}}^T(\widetilde y)$ intersects the same $5$ deck images of $\widetilde \phi_0$ as above, in the same order.
 \end{enumerate}

Since $\widetilde x_0$ is nonwandering, we can find $\widetilde y_\ast\in \widetilde V$ and some integer $M>T+2$, 
such that
$\widetilde f^M(\widetilde y_\ast)$ is also in in $\widetilde V$. 
By the choice of $\widetilde V$, the paths $\widetilde I^{T+2}_{\widetilde {\mathcal F}}(\widetilde f^{-1}(\widetilde y_\ast))$ and $\widetilde I^{T+2}_{\widetilde {\mathcal F}}(\widetilde f^{M-1}(\widetilde y_\ast))$ both contain subpaths  that are $\widetilde {\mathcal F}$-equivalent to $\widetilde I^{T}_{\widetilde {\mathcal F}}(\widetilde x_0)$ and that therefore meets the same five deck images of $\widetilde \phi_0$, in the same order. 
We deduce that the transverse path $\widetilde I_{\widetilde{\mathcal F}}^{M+T+1}(\widetilde f^{-1}(\widetilde y_\ast))$ 
intersects each translate $u_i(\widetilde \phi_0)$ at least twice. Then, with these properties, 
via an
appropriate re-parametrization, 
we obtain an admissible $\widetilde {\mathcal F}$-transverse path $\widetilde \beta$ satisfying all the conditions of Proposition~\ref{new_prop}, and the conclusion follows.
\end{proof}
In order to show Theorem~\ref{non-wandering_Bounded_M}, 
let us prove a covering lemma first. 
\begin{lemma}\label{local_good_neighbourhood}
Under the hypotheses of Theorem~\ref{non-wandering_Bounded_M}, there exists a bounded fundamental domain $\widetilde S_0$ and a finite family 
of bounded open disks $\{ \widetilde E_j\}_{j=1}^{\kappa}$
in $\widetilde S$, satisfying the following properties. 
\begin{enumerate}
\item The union $\bigcup_{j=1}^{\kappa} \widetilde E_j$ projects to a fully essential open set $O$, whose complement 
$O^c$ in $S$ contains $\text{Fix}(I)$. 
\item Any $\widetilde f$-orbit segment intersecting both $\widetilde S_0$ and its complement in $\widetilde S$, 
must intersect some $\widetilde E_j$ from the above list.
\item For any $j=1,\cdots, \kappa$, there exists some leaf $\widetilde \phi_j$ of $\widetilde{\mathcal{F}}$ such that, for any $\widetilde z_0\in \widetilde E_j$, the transverse path $\widetilde I^2_{\widetilde {\mathcal F}} (\widetilde f^{-1}(\widetilde z_0))$ intersects $\widetilde \phi_j$.
\end{enumerate}
\end{lemma}

\begin{proof} By the assumption that $\text{Fix}(f)$ is inessential, we also know that  $\text{Fix}(I)$ is a closed and inessential set. This means that $\text{Fix}(I)$ is contained in the interior of an open topological disk $U_0$, and by compactness we can find another topological disk $U$ containing $\text{Fix}(I)$ such that $\overline U \subset U_0$. We can, for each $z$ in  $\text{Fix}(I)$, find a small open disk $D_z\subset U$ such that $\overline{ f(D_z)}$ is also contained in $U$, and by compactness one can find points $(z_i)_{i\in\{1, \hdots, n_0}$ in $\text{Fix}(I)$ such that $\bigcup_{i=1}^{n_0} D_{z_i} \supseteq \text{Fix}(I)$. Let $\widetilde U$ be a connected component of the lift of $U$,
and, for each $z_i$, let $\widetilde D_{z_i}$ be the connected component of the lift of $D_{z_i}$ that is 
contained in $\widetilde U$  and let $\widetilde z_i$ be the lift of $z_i$ that is contained in $\widetilde D_{z_i}$. 
Note that, as each $z_i$ is in $\text{Fix}(I)$, each $\widetilde z_i$ is in $\text{Fix}(\widetilde f)$, and therefore $\widetilde f(\widetilde D_{z_i})$ intersects $\widetilde D_{z_i}$ and therefore intersects $\widetilde U$. 
Since $f(D_{z_i})\subset U$, we deduce that $\widetilde f(\widetilde D_{z_i})\subset \widetilde U$. 
Therefore, if a point $\widetilde y_0$ belongs to $\widetilde U$ but its image does not, we have that $\widetilde y_0$ cannot belong to $\bigcup_{i=1}^{n_0} \widetilde D_{z_i}$. Let $\widetilde S_0$ be a fundamental domain for $S$ that contains $\widetilde U$, and since $\overline U$ is
contained in a topological disk, we may assume that $\widetilde S_0$ is bounded. 
Write $\widetilde K=\widetilde S_0 \setminus \bigcup_{i=1}^{n_0} \widetilde D_{z_i}$, then any point in $\widetilde S_0$ whose image lies outside $\widetilde S_0$ must lie in $\widetilde K$.

Since each point in $\text{Fix}(I)$ has a lift in the interior of $\widetilde S_0 \backslash \widetilde K$, we get that $\overline{\widetilde K}$ is a subset of $\text{Dom}(\widetilde{\mathcal F})$. This implies that, by Lemma~\ref{lemma_of_continuity}, for each $\widetilde y \in \overline{\widetilde K}$, there exists an open topological disk $\widetilde E_{\widetilde y}$  containing $\widetilde y$, such that if $\widetilde z_0 \in \widetilde E_{\widetilde y}$, then $\widetilde I^2_{\widetilde {\mathcal F}} (\widetilde f^{-1}(\widetilde z_0))$ intersects $\widetilde \phi_{\widetilde y}$. The result now follows from the compactness of  $\overline{\widetilde K}$.
\end{proof}

\begin{proof}[Proof of Theorem~\ref{non-wandering_Bounded_M}] Obtain $\widetilde S_0$, $\kappa$ and $\widetilde E_j, 1\le j\le \kappa$ from the previous Lemma. Let $L_0$ be the diameter of $\widetilde S_0$, and let $L_1=\max_{\widetilde z\in\widetilde S} \|\widetilde f(\widetilde z)-\widetilde z\|$, which is well defined since $\widetilde f$ commutes with the deck transformations. 
Suppose statement (2) does not happen. Then we choose $M \geq (4\kappa+2)(L_0+L_1)$, and for 
some $\widetilde z_0\in \Omega(\widetilde f)$, and some positive integer $N_1$, $\|\widetilde f^{N_1}(\widetilde z_0) - \widetilde z_0\| > M$. Then one can find $(4\kappa+2)$  integers $k_l$, with $0\le k_1<k_2<\hdots<k_{4\kappa+2}\le N_1$ such that $\|\widetilde f^{k_{l_1}}(\widetilde z_0)- f^{k_{l_2}}(\widetilde z_0)\|>L_0$ if $l_1<l_2$ are two elements chosen from $\{1,\cdots, 4\kappa +2\}$. 
 Therefore one can find distinct deck transformations $u_l, 1\le l\le (4\kappa+2)$ such that $\widetilde f^{k_{l}}(\widetilde z_0)\in u_l(\widetilde S_0)$.

Again by the previous lemma, one finds $(j_l)_{1\le l\le 4\kappa+1}$ and $(k'_l)_{1\le l\le 4\kappa+1}$, with $k_l\le k'_l<k_{l+1}$ such that $\widetilde f^{k'_{l}}(\widetilde z_0)\in u_l(\widetilde E_{j_l})$. So by pigeonhole principle, there exists some disk $\widetilde E$ from the list $\{\widetilde E_j\}_{j=1}^\kappa$, such that the trajectory $\widetilde I^{[0,N_1]}_{\widetilde{\mathcal{F}}}(\widetilde z_0)$  intersects at least five distinct translates of $\widetilde E$. By Corollary~\ref{five_translates_we_will_use}, it follows that $f$ admits non-contractible periodic orbits, which is statement (1). 
So the proof is finished. 
  \end{proof}

\subsection{Proving Proposition~\ref{new_prop}}\label{five_proof}
This subsection is devoted to the proof of 
Proposition~\ref{new_prop}. 
Let us now make some preparations. 
We assume that $S$ is connected, otherwise we work with the connected component of $S$ that contains $\widetilde \phi_0$.
Without loss of generality and to lighten the notation, we assume that $u_0$ is the identity, otherwise we just rename $u_0(\widetilde \phi_0)$ as $\widetilde \phi_0$. 
Let $\widetilde \gamma: [0,1]\to \text{Dom}(\widetilde {\mathcal F})\subset \widetilde S$  be an $\widetilde {\mathcal F}$-transverse path such that $\widetilde \gamma(0)\in \widetilde \phi_0$ and $\widetilde \gamma(1)=u_1 (\widetilde \gamma(0)) \in u_1(\widetilde \phi_0)$, 
where $u_0$ and $u_1$ are two distinct deck transformations, and such that the projection of $\widetilde \gamma\mid_{(0,1)}$ does not intersect the projection of $\widetilde \phi_0$. Then, we can look at the quotient 
$\check S= \widetilde S/u_1$, which is homeomorphic to the open annulus.
Denote by $\widecheck{\gamma}$ the projection of the path 
$\widetilde \gamma\mid_{[0,1]}$ to $\check S$, and observe that it is $\widecheck{\mathcal {F}}$-equivalent to a closed $\widecheck{\mathcal{F}}$-transverse loop, where $\widecheck{\mathcal{F}}$ is the projection of $\widetilde{\mathcal{F}}$ to $\check S$.  
We also assume that $\widetilde \gamma$ has no $\widetilde {\mathcal{F}}$-transverse intersection with $u_1^{j}(\widetilde \gamma)$ for all $j\in \Bbb Z$ (in particular, $\widetilde \gamma$ has no transverse self-intersection). 

Then we extend the domain of 
$\widetilde \gamma$ to the whole line $\Bbb R$ by defining 
\begin{equation}
\widetilde \gamma(t+1)=u_1(\widetilde \gamma(t)) \text{ for all } t \in \Bbb R.
\end{equation}
As $\widetilde S$ is homeomorphic to disjoint copies of the plane, we recall that a line is a proper and injective map from $\Bbb R$ to $\widetilde S$. 
\begin{lemma}
$\widetilde \gamma$ is a line. 
\end{lemma}
\begin{proof}
That $\widetilde \gamma$ is proper follows directly from the fact that $u_1$ is a non-trivial deck transformation. 
We want to show that $\widetilde \gamma$ is injective. 
From standard Brouwer theory we know that if $\widetilde \gamma\mid_{[0,1]}$
 is a translation arc, that is, if $\widetilde \gamma\mid_{[0,1)}$ is disjoint from $u_1(\widetilde \gamma\mid_{[0,1)})$, 
then it holds that  
$(u_1)^{n_1}(\widetilde \gamma\mid_{[0,1)})$ is disjoint from 
$(u_1)^{n_2}(\widetilde \gamma\mid_{[0,1)})$ whenever $n_1\not=n_2$. Therefore we need only to consider three cases:
\begin{itemize}
\item[(a)]There exists $0<t_0<t_1<1$ such that $\widetilde \gamma(t_0)=\widetilde\gamma(t_1)$.
\item[(b)] There exists $0<t_0<t_1<1$ such that $\widetilde \gamma(t_0)=u_1(\widetilde\gamma(t_1))$.
\item[(c)] There exists $0<t_0<t_1<1$ such that $\widetilde \gamma(t_0)=(u_1)^{-1}(\widetilde\gamma(t_1))$.
\end{itemize}
In case $(a)$, 
 one finds that there exists some $t_0\le s_0<s_1\le t_1$ such that $\widetilde \gamma(s_0)$ 
 and $\widetilde \gamma(s_1)$ lie in the same leaf of $\widetilde{\mathcal{F}}$,
 but such that for all $s_0\le s<s'<s_1$ one has that $\widetilde \gamma(s)$ and $\widetilde \gamma(s')$ belong to different leaves. 
This implies that $\widetilde \gamma\mid_{[s_0,s_1]}$ is $\widetilde {\mathcal F}$-equivalent to a simple closed loop $\Gamma$, and since $\Gamma$ does not intersect any translate of $\widetilde \phi_0$, one has that there exists a drawing interval $J\subset[0,1]$ for $\Gamma$. In case both $\widetilde\gamma(0)$ and $\widetilde \gamma(1)$ lie in the same connected component of the complement of $\Gamma$, 
we have a contradiction using Lemma~\ref{drawing_crossing_transverse}, because we assumed $\widetilde \gamma$ has no $\widetilde{\mathcal{F}}$-transverse self-intersection.  
If not, and we assume that $\widetilde \gamma(1)$ is in the bounded connected component of $\Gamma^c$, then there exists a smallest $n>1$ such that $\widetilde \gamma(n)$ is in the unbounded connected component of $\Gamma^c$, which is finite since $\widetilde \gamma$ is unbounded. 
In this case $\widetilde \gamma$ has a crossing component contained in $[n-1, n]$. In other words, 
by Lemma~\ref{drawing_crossing_transverse}, 
$\widetilde \gamma\mid_{[0,1]}$ has an $\widetilde{\mathcal{F}}$-transverse intersection with $\widetilde \gamma\mid_{[n-1,n]}= (u_1)^{n-1}\widetilde \gamma\mid_{[0,1]}$, 
which again contradicts our assumptions. 

In case (b),  proceeding as in the above paragraph we find $t_0\le s_0<s_1\le t_1$ such that $\widecheck \gamma \mid_{[s_0, s_1]}$ is $\widecheck{\mathcal{F}}$-equivalent to a simple transverse loop $\widecheck \Gamma$ which is essential in $\widecheck S$, and which does not intersect the leaf $\widecheck \phi_0$, the projection of $\widetilde \phi_0$.  We can apply here again Lemma~\ref{drawing_crossing_transverse} to obtain that $\widecheck \gamma \mid_{[0,1]}$ has a $\widecheck{\mathcal{F}}$-transverse self-intersection, which implies that $\widetilde \gamma \mid_{[0,1]}$ must have an $\widetilde{\mathcal{F}}$-transverse intersection with $(u_1)^{l}(\widetilde \gamma \mid_{[0,1]})$ for some $l\in \Bbb Z$, again a contradiction. Case (c) is done almost exactly as case (b).
\end{proof}

\begin{definition}\label{A_0and_strip}
Suppose $\widetilde \gamma, u_1,\widetilde \phi_0$ are as in the previous paragraph. 
Let $A_0=A_0(\widetilde \gamma)$ denote
the union of all leaves met by $\widetilde \gamma \mid_{[0,1]}$, and call it the \emph{foliated block} for $\widetilde\gamma([0,1])$, or simply the \emph{block}. For each $k\in \Bbb Z$, call the $k$-th translated block as follows
\begin{equation}\label{definition_of_Ak}
A_k= A_k(\widetilde \gamma)= \big(u_1\big)^k (A_0).
\end{equation} 
Now define 
\begin{equation}\label{definition_infinite_stripA}
A =A(\widetilde \gamma)=\bigcup_{k\in\Bbb Z}  A_k = \bigcup_{k\in \Bbb Z}\big(u_1 \big)^k (A_0).
\end{equation} In other words, 
$A(\widetilde \gamma)$ is the infinite strip consisting of all the leaves intersected by $\widetilde \gamma \mid_{\Bbb R}$.
Denote by $\mathcal L(A)$ the union of all the connected components of the complement of $A$ that lie in the left of $\widetilde \gamma$. Similarly, let $\mathcal R(A)$ denote the union of all connected components of the complement of $A$ that lie in the right of $\widetilde \gamma$.
\end{definition}

If an $\widetilde {\mathcal F}$-transverse path $\widetilde \beta \mid_{[a,a+\varepsilon]}$ satisfies that 
  $\widetilde \beta(a)\in \mathcal L(A)$ (respectively, $\mathcal R(A)$)
   and $\widetilde \beta \big ((a,a+\varepsilon]\big) \subset A_k$ for some $k\in\Bbb Z$, 
  then we say the segment 
  $\widetilde \beta \mid_{[a,a+\varepsilon]}$ enters $A$ from $\mathcal L(A_k)$ 
  (respectively, from $\mathcal R(A_k))$. Similarly, 
  we say $\widetilde \beta \mid_{[a-\varepsilon,a]}$ leaves $A$ from 
  $\mathcal L(A_k)$ (respectively, $\mathcal R(A)$), 
  if $\widetilde \beta \big([a-\varepsilon,a) \big)\subset A_k$ and $\widetilde \beta (a)\in \mathcal L(A)$ 
  (respectively, $\mathcal R(A)$).
  
\begin{definition}\label{above_below_mixed}
Suppose two $\widetilde {\mathcal F}$-transverse paths
$\widetilde \gamma$ and $\widetilde \gamma': [0, 1]\to \text{Dom}(\widetilde {\mathcal F})$ are 
as described at the beginning of this subsection, and 
they satisfy that 
$\widetilde \gamma(0)=\widetilde \gamma'(0)\in \widetilde \phi_0$ and 
$\widetilde \gamma(1)=\widetilde \gamma'(1)= u_1(\widetilde \gamma(0)) \in u_1(\widetilde \phi_0)$,
where $u_1$ is a nontrivial deck transformations. 
Then one can obtain the blocks $A_0=A_0(\gamma)$ and $A_0'=A_{0}(\gamma')$ 
as in Definition~\ref{A_0and_strip} with respect to the paths $\widetilde \gamma$ and $\widetilde \gamma'$, respectively. 
Also define the translated blocks $A_k$ and $A_k'$ as in (\ref{definition_of_Ak}) for all $k\in \Bbb Z$, 
as well as the infinite strips $A=A(\gamma)$ and $A'=A(\gamma')$ as in (\ref{definition_infinite_stripA}), with respect to 
the extended infinite paths $\widetilde \gamma$ and $\widetilde \gamma'$, respectively. 
We say $A'$ is below $A$ if $A'\subset \mathcal R(A) \bigcup A$. 
We say $A'$ is above $A$ if $A' \subset \mathcal L(A) \bigcup A$.
If neither case above happens, we say $A'$ and $A$ are in mixed position. 
\end{definition}

See Figure~\ref{Figure_1} for an illustration of the case of $A'$ below $A$. In this case, the closure of the 
bounded region $D$ is the bounded connected component of the complement of $A_0\bigcup A_0'$.
\begin{figure}[h]
	\centering
	\includegraphics[width=0.8\textwidth]{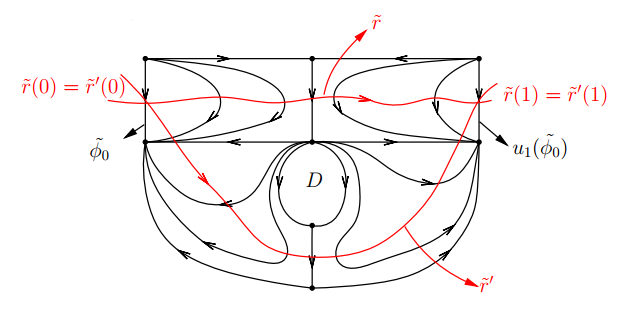}
	\caption{Here is the case where $A'$ is below $A$. 
	We only depict the 
	blocks $A_0$ and $A_0'$ above. The leaves of the foliation are in black and 
	the transverse paths $\widetilde \gamma$ and $\widetilde \gamma'$ are in red.  
	}\label{Figure_1}
\end{figure}

 Now we start the proof of Proposition~\ref{new_prop}, by making several first reductions. 
Claim that without loss of generality, 
one can assume $\widetilde \beta \mid_{[t_0,t_1]}$ has no $\widetilde {\mathcal F}$-transverse self-intersection. 
Suppose otherwise, that is, $\widetilde \beta \mid_{[t_0,t_1]}$ 
 has an $\widetilde {\mathcal F}$-transverse self-intersection at times $t$ and $t'$, with $t_0<  t<t' < t_1$. Then we can apply Lemma~\ref{shortening_paths} 
 to replace the transverse path by a new transverse path, still denoted by 
 $\widetilde \beta$, whose restriction to $[t_0,t_1]$ shortens and avoids 
$\widetilde {\mathcal F}$-transverse self-intersections, and moreover
$\widetilde \beta$ satisfies the rest of the conditions in Proposition~\ref{new_prop}. 

Similarly, using forcing property, we can also assume that 
$\widetilde \beta \mid_{[t_0,t_1]}$ and $\widetilde \beta \mid_{[1+t_0,1+t_1]}$
have no $\widetilde {\mathcal F}$-transverse intersection. 
To see this, suppose 
an $\widetilde {\mathcal F}$-transverse intersection between the above two transverse paths happens 
at the time $t\in [t_0,t_1]$ and $t' \in [1+t_0,1+t_1]$, respectively. 
Then we apply Lemma~\ref{Forcing_proposition} to obtain a new 
$\widetilde{\mathcal F}$-transverse admissible path, which concatenates $\widetilde \beta \mid_{[0,1+t']}$
with $\widetilde \beta \mid_{[t,t_4+\varepsilon]}$, which, after appropriate re-parametrization, 
avoids $\widetilde{\mathcal F}$-transverse 
intersection and still satisfies the original properties. 

Then we assume, for a contradiction, that $\widetilde \beta$ has no $\widetilde f$-transverse intersection with $u(\widetilde\beta)$ for any deck transformation $u$ other than the identity, because otherwise we already know there exists non-contractible periodic points and we have nothing to prove.

With all the above reductions, the path $\widetilde \beta \mid_{[t_0,t_1]}$ 
is $\widetilde {\mathcal F}$-equivalent to 
some path $\widetilde \gamma: [0,1]\to \text{Dom}(\widetilde {\mathcal F})$, with 
$\widetilde \gamma(1)=u_1 (\widetilde \gamma(0))$ and $\widetilde \gamma$ has the same properties as 
we have assumed in earlier discussion at the beginning of this sub-section. 
Likewise, the path $\widetilde \beta \mid_{[1, 1+t_1]}$ 
is $\widetilde {\mathcal F}$-equivalent to 
some path $\widetilde \gamma': [0,1]\to \text{Dom}(\widetilde {\mathcal F})$, with $\widetilde \gamma'(0)=\widetilde \gamma(0)$ and 
$\widetilde \gamma'(1)= \widetilde \gamma(1)= u_1 (\widetilde \gamma'(0))$. 
Then by Definition~\ref{above_below_mixed}, they induce the foliated 
blocks $A_0= A_0(\widetilde \gamma), A_0'= A_0(\widetilde \gamma')$,
and moreover the infinite strips $A, A'$, respectively. Note that if $A$ and $A'$ are in mixed position, 
then it means $\widetilde \beta \mid_{[t_0,t_1]}$ and $\widetilde \beta \mid_{[1+t_0,1+t_1]}$ intersect 
$\widetilde {\mathcal F}$-transversely. 
By previous reductions, we have assumed it is not the case. 
Thus, let us suppose $A'$ is below $A$ and the other case is symmetric.
Then we will look at a bigger region defined as follows.  

\begin{definition}
Define $U$ to be the union of $A$ and $A'$ and 
  all the closures of the bounded connected components of 
  the complement of $\widetilde \gamma \mid_{[k,k+1]}\bigcup \widetilde \gamma' \mid_{[k,k+1]}$ for all  $k\in\Bbb Z$. 
  \end{definition}
 Viewing Figure~\ref{Figure_1}, we see $U$ is simply the union of $A_0,A_0'$, the closure of $D$ and all their $u_1^j$ translates.  
  
  \begin{lemma}\label{outside_strip}
There exists some $b \in (t_1, 1)$ such that $\widetilde \beta(b) \notin U$.
\end{lemma}
\begin{proof}
Suppose for contradiction that $\widetilde \beta \mid_{[t_1,1]}\subset U$. 
The leaf $\widetilde \phi_0$ divides $U$ into two disconnected subsets, 
namely $U\setminus \widetilde \phi_0= L\cup R$.
Clearly, since $\widetilde \beta \mid_{[0,1]}$ is an $\widetilde {\mathcal F}$-transverse path, it follows that 
$\widetilde \beta \mid_{(0,1]} \subset R$, which is absurd, because $\widetilde \beta (0)=\widetilde \beta(1)$.
\end{proof}
After the Lemma, we can write
\begin{align}
b_0 & = \inf \{ b>t_1 \big | \widetilde \beta(b)\notin U\}.\\
c_0 & = \sup \{c<1 \big | \widetilde \beta(c)\notin U\}.
\end{align}
Note that if a path leaves or enters $U$, it either does this from $\mathcal L(A)$ or from 
$\mathcal R(A')$, due to the assumption that $A'$ is below $A$. 
Thus, it suffices to discuss the following cases. 
\begin{enumerate}[\text{Case} (1).]
\item $\widetilde \beta (b_0) \in \mathcal L(A), \widetilde \beta(c_0) \in \mathcal L(A)$.
\item $\widetilde \beta (b_0) \in \mathcal L(A), \widetilde \beta(c_0) \in \mathcal R(A')$.
\item $\widetilde \beta (b_0) \in \mathcal R(A'), \widetilde \beta(c_0)\in \mathcal L(A)$.
\item $\widetilde \beta (b_0) \in \mathcal R(A'), \widetilde \beta(c_0) \in \mathcal R(A')$.
\end{enumerate}

\begin{lemma}\label{intersection_both_sides}
In all cases, either there is some integer $\ell\not=0$ such that $\widetilde \beta$ and $\big( u_1 \big)^\ell \cdot \widetilde \beta $ intersect $\widetilde {\mathcal{F}}$-transversely, or there exists some $t_1<d<d^\ast<1$ such that 
such that $\widetilde \beta \mid_{[t_0,t_1]}$ and 
$\widetilde \beta \mid_{[d,d^\ast]}$ intersect $\widetilde {\mathcal{F}}$-transversely and
 also $\widetilde \beta\mid_{[1+t_0,1+t_1]}$ and 
$\widetilde \beta \mid_{[d,d^\ast]}$ intersect $\widetilde {\mathcal{F}}$-transversely. 
Furthermore, we can assume that $t_i\notin [d,d^\ast]$ for $i \in \{2,3,4\}$. 
\end{lemma}

\begin{proof}
We deal with these cases separately. \\
\noindent {\textbf{Case (1)}}.
$\widetilde \beta (b_0) \in \mathcal L(A), \widetilde \beta(c_0) \in \mathcal L(A)$. (See Figure~\ref{Figure_2} for an
illustration of this case.)\\
We can suppose for some $\varepsilon>0$, 
$\widetilde \beta\mid_{[b_0-\varepsilon, b_0]}$
leaves $A$ from $\mathcal L(A_{n_1})$ for some $n_1\geq 1$.
$\widetilde \beta \mid_{[c_0,c_0+\varepsilon]}$ enters $A$ from 
$\mathcal L(A_{n_0})$ for some $n_0< 0$. Let  $\ell=n_1 -n_0 -1 \geq 1$. 
Then $ (u_1 )^\ell \cdot \widetilde \beta \big( [c_0, 1+t_1] \big)$ 
enters $A$ from $\mathcal L(A_{n_1-1})$ and stays in $U$ until reaching 
$\big(u_1 \big)^{n_1+1}(\widetilde \phi_0)$.
On the other hand, $\widetilde \beta \mid_{[t_0, b_0]}$ starts at $u_0(\widetilde \phi_0)$ and remains in $U$ intersecting $\big(u_1 \big)^{i}(\widetilde \phi_0), 0<i\le n_1,$ before it 
leaves $A$ from $\mathcal L(A_{n_1})$. Therefore, 
both $\widetilde {\mathcal F}$-transverse paths $ \big(u_1\big)^\ell \cdot \widetilde \beta \big ( [c_0, 1+t_1] \big)$ 
and $\widetilde \beta \mid_{[t_0, b_0]}$ 
intersect the leaf 
$ \big(u_1\big)^{n_1} (\widetilde \phi_0)$. By Lemma~\ref{drawing_crossing_transverse},
%\textcolor{black}{CHACAR AQUI DE NOVO}, 
they must have an $\widetilde {\mathcal F}$-transverse intersection.\\
\begin{figure}[h]
	\centering
	\includegraphics[width=0.8\textwidth]{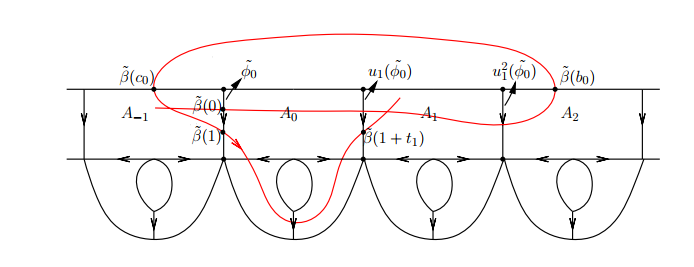}
	\caption{Here depicted is one possible example of Case (1). We have $n_1=2,n_0=-1$, and therefore $\ell=2$. 
	}\label{Figure_2}
\end{figure}

\noindent {\textbf{Case (2)}}. $\widetilde \beta (b_0) \in \mathcal L(A), \widetilde \beta(c_0) \in \mathcal R(A')$.\\
Note in this case there exists a crossing component for $U$.
More precisely, there exists an open interval $(d,d^\ast)\subset (b_0,c_0)$, 
such that, $\widetilde \beta(d)\in \mathcal L(A), \widetilde \beta (d^\ast) \in \mathcal R(A')$,
and $\widetilde \beta \big((d,d^\ast) \big) \subset U$. 
Similar with previous case, 
there are integers $n_1\geq 1, n_0\leq -1$,
 such that $\widetilde \beta \mid_{[t_0,b_0]}$ leaves $A$ from $\mathcal L(A_{n_1})$, 
 and $\widetilde \beta \mid_{[c_0,1]}$ enters $A'$ from $\mathcal R(A_{n_0}')$.
There is also $n_2\in \Bbb Z$ such that for some $\varepsilon>0$, 
$\widetilde \beta \mid_{[d,d+\varepsilon]}$ enters $U$ from $\mathcal L(A_{n_2})$, 
and note that $\widetilde \beta \mid_{[d,d^\ast]}$ leaves $U$ from $\mathcal R(A')$.
 
 %< s_1 \leq n_1+1$, 
%such that $\widetilde \gamma\big |_{[0, s_1)}$ and 
%$\widetilde \beta_{[t_0,b_0)}$ are $\widetilde {\mathcal F}$-freely holonomically homotopic.
%Moreover, we also find $n_2\leq T_d <n_2 + 1$, and $T_d<T_{d'}$, 
 %such that $\widetilde \gamma\big |_{(T_d, T_{d'})}$
%and 
%$\widetilde \beta \big |_{(d, d')}$ are $\widetilde {\mathcal F}$-freely holonomically homotopic. 
Similar to the argument in the final part of Case (1), 
we know that 
if $0\le n_2+ \ell<n_1$ for some $\ell \in \Bbb Z$, then
\begin{equation}\label{ell_translate}
\big(u_1\circ u_0^{-1}\big)^{\ell} \cdot \widetilde \beta \mid_{[d, d^\ast]} \text{ and } 
\widetilde \beta \mid_{[t_0, b_0]} \text{  intersect } \widetilde {\mathcal F}\text{-transversely}. 
\end{equation}
We further consider three subcases.\\
 \noindent{ \textbf{subcase (2.1)}}: $n_1\geq 2$. In this case, depending on $n_2 \geq 1$ or $n_2\leq 0$, 
 one can choose either a strictly negative or a strictly positive integer $\ell$. 
 It follows statement (\ref{ell_translate}) holds for non-zero $\ell$.\\  
 \noindent{ \textbf{subcase (2.2)}}: $n_2 \neq 0$. In this case, we simply choose $\ell=-n_2$. 
  Then statement (\ref{ell_translate}) holds for non-zero $\ell$.\\
   \noindent{ \textbf{subcase (2.3)}}: 
   $n_2=0$ and $n_1=1$. (See Figure~\ref{Figure_3} for an illustration of this subcase.)
    Now there is $\ell \geq n_2=0$ and some $\varepsilon>0$, 
   such that $\widetilde \beta \mid_{[d^\ast-\varepsilon,d^\ast]}$ leaves $A'$ from $\mathcal R(A_\ell')$.
   If $\ell\geq 1$, then similar to Case (1), 
   we see $(u_1\circ u_0^{-1})^\ell \cdot \widetilde \beta \mid_{[c_0,1+t_1]}$
   and $\widetilde \beta \mid_{[d,d^\ast]}$ intersect $\widetilde {\mathcal F}$-transversely. 
   If $\ell=0$, then $\widetilde \beta \mid_{[d,d^\ast]}$ and  $\widetilde \beta \mid_{[t_0,t_1]}$ intersect $\widetilde {\mathcal F}$-transversely. But in this case, 
      $\widetilde \beta \mid_{[d,d^\ast]}$ and  $\widetilde \beta \mid_{[1+t_0,1+t_1]}$ also 
      intersect $\widetilde {\mathcal F}$-transversely. Note that in this case, since $\widetilde\beta \mid_{[d,d^\ast]}$ lies int he union of $A_0\cup A'_0$ with the bounded connected components of its complement, and since neither $\widetilde\beta \mid_{[t_0,t_1]}$ nor $\widetilde\beta \mid_{[1+t_0,1+t_1]}$ intersect any copy of $\widetilde \phi_0$, we deduce that $t_i\notin [d, d^\ast], i\in\{2,3,4\}$, ending the proof in this case. \\
 \begin{figure}[h]
	\centering
	\includegraphics[width=0.8\textwidth]{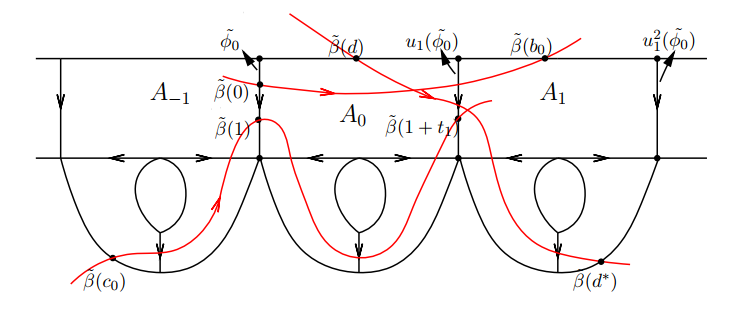}
	\includegraphics[width=0.8\textwidth]{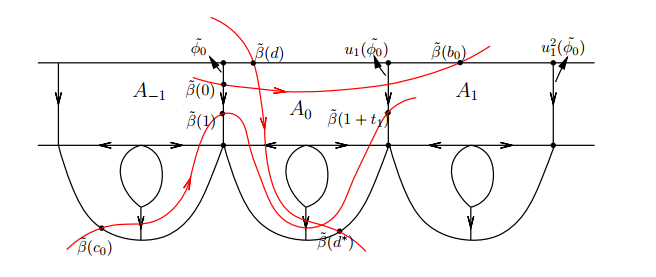}
	\caption{Here depicted are two possible situations of subcase (2.3). 
	We have $n_1=1,n_0=-1,n_2=0$. In the figure in the first line, $\widetilde \beta \mid_{[d,d^\ast]}$ leaves 
	$U$ from $\mathcal R(A_1')$, while in the figure in the second line, 
	$\widetilde \beta \mid_{[d,d^\ast]}$ leaves 
	$U$ from $\mathcal R(A_0')$.
	}\label{Figure_3}
\end{figure}

  \noindent {\textbf{Case (3)}}. $\widetilde \beta (b_0) \in \mathcal R(A), \widetilde \beta(c_0)\in \mathcal L(A)$.\\
  This case is very similar with Case (2) and we omit the proof.\\
    
  \noindent {\textbf{Case (4)}}. $\widetilde \beta (b_0) \in \mathcal R(A), \widetilde \beta(c_0) \in \mathcal R(A)$.\\
  This case is very similar to Case (1) and we omit the proof.
    \end{proof}

\begin{proof}[End of the proof of the Proposition~\ref{new_prop}] 
After Lemma~\ref{intersection_both_sides}, if we are in the first situations and we find some integer $\ell\not=0$, such that 
$\widetilde \beta$ and 
$\big(u_1\big)^\ell \cdot \widetilde \beta $ intersect $\widetilde {\mathcal F}$-transversely,
then we can conclude with non-contractible periodic orbit, 
with the help of Lemma~\ref{translate_intersection}. 
If not, and if $d$ given by Lemma~\ref{intersection_both_sides} is larger than $t_3$, then the $\widetilde{\mathcal F}$-transverse paths
$\widetilde \beta \mid_{[0, t_1]}$ and $\widetilde \beta \mid_{[d, d^\ast]}$ intersect 
$\widetilde {\mathcal F}$-transversely at moments 
$T_1<t_1$ and $T_1'>t_3$.  In this case the loop $\widetilde \beta\mid_{[T_1, T_1']}$ is admissible and intersects $u_1(\widetilde \phi_0), u_2(\widetilde \phi_0)$ and $u_3(\widetilde \phi_0)$. Therefore we can apply Lemma~\ref{three_transformations} and get the result. Finally, we are left with the case 
when $d<t_3$, which implies that $d^{\ast}<t_3$. 
Then the $\widetilde {\mathcal F}$-transverse paths
$\widetilde \beta \mid_{[1+t_0, 1+t_1]}$ and $\widetilde \beta \mid_{[d, d^\ast]}$ intersect 
$\widetilde {\mathcal F}$-transversely at moments $T_2'>1$ and $T_2<t_3$. 
In this case the loop $\widetilde \beta \mid_{[T_2, T_2']}$ is admissible and
it intersects $u_3(\widetilde \phi_0), u_4(\widetilde \phi_0)$ and $u_0(\widetilde \phi_0)$ and we end the proof with Lemma~\ref{three_transformations} as before.
\end{proof}

\section{Unbounded Deviation and Unbounded Orbit}
In this section, we show Proposition~\ref{unbounded_trajectory}, 
which guarantees existence of orbits unbounded in certain direction, under some mild assumptions. 
We start with the following statement, originally proved in \cite{Tal_noncontractible}, 
which is now seen as a corollary of Theorem~\ref{non-wandering_Bounded_M}. 
\begin{coro}\label{Tal_noncontractible}
Let $\widetilde f \in \widetilde {\text{Homeo}}^+_{0,\text{nw}}(\Bbb T^2)$, and assume 
$(0,0)\in \rho(\widetilde f)$.
Then one of the following cases is true.
\begin{enumerate}
\item $\pi \big( \text{Fix}(\widetilde f) \big)$ is essential. 
\item $f$ has non-contractible periodic points.
\item There exists some $M>0$, with the following properties. 
For any $x\in \Bbb T^2$,
 which is either an $f$-periodic point, or is $f$-inessential,
 the orbit $\{\widetilde f^n(\widetilde x)\}_{n\in \Bbb Z}$ has diameter bounded from above by $M$, 
for any lift $\widetilde x$ of $x$. 
\end{enumerate}
\end{coro}
%\marginpar{\textcolor{blue}{We have defiend inessential points, perhaps we can use the definition? But otherwise we should be more clear on what contained in a periodic disk means.}}
\begin{proof}
Let us 
assume the first two items do not happen. Then if $x$ is periodic for $f$, any lift $\widetilde x$ of $x$ is also periodic and therefore non-wandering. On the other hand, if $x$ is inessential, since $f$ is non-wandering we have that $x$ is contained in a an $f$ periodic open disk $D$, and if $k$ is the period of $D$, then we can assume $f^{j}(D)$ is disjoint from $D$ for $i\le j\le k-1$. Of course since $f^k$ preserves $D$ and it is nonwandering, Brouwer theory tell us there is a fixed point for $f^k$ in $D$. This implies that, if $\widetilde D$ is a connected component of $\pi^{-1}(D)$, then $\widetilde f^{k}(\widetilde D) = \widetilde D+ v$ for some $v\in \Bbb Z^2$, and since there was a periodic point in $\widetilde D$ which is contractible, we get that $v$ is null. This implies, since $x$ is non-wandering, that  $\widetilde x$ belongs to 
$\Omega(\widetilde f)$. Then item (3) follows immediately from Theorem~\ref{non-wandering_Bounded_M}. 
\end{proof}

\begin{prop}\label{unbounded_trajectory} 
Given $\widetilde f\in \widetilde {\text{Homeo}}^+_{0,\text{nw}}(\Bbb T^2)$, whose 
rotation set $\rho(\widetilde f)$ contains $(0,0)$. 
Assume that
for some $v\in \Bbb S^1$, 
\begin{equation}\label{unbounded assumption}
\sup_{\widetilde x\in \Bbb R^2, n\geq 0} \text{pr}_v(\widetilde f^n(\widetilde x)-\widetilde x)= +\infty.
\end{equation}
Then, one of the followings is true.
\begin{enumerate}[(a)] 
\item $ \pi(\text{Fix}(\widetilde f))$ is essential.
\item There is some $f$-invariant annulus, 
whose homological direction is not perpendicular to $v$.
\item $f$ has non-contractible periodic points.
\item There exists $\widetilde {x_0} \in \Bbb R^2$, such that,
\begin{equation}
\sup_{n\geq 0} \text{pr}_v(\widetilde f^n(\widetilde x_0)-\widetilde x_0)=+\infty.
\end{equation} 
\end{enumerate}
\end{prop}
\begin{proof}  
Fix the direction $v\in \Bbb S^1$. 
For any $K>0$, denote
\begin{equation}
B_K=
\{x \big| \text{pr}_v 
(\widetilde f^n(\widetilde x)-\widetilde x) \leq K \text{ for any } n\geq 0 \text{ and any lift }\widetilde x \text{ of } x\}.
\end{equation}
Then $B_K$ is a closed subset for any $K>0$. 
Now we denote its interior by 
$I_K=\text{Int}(B_K)$, possibly an empty set, 
and then denote $U=\bigcup_{K>0} I_K$.

From now on, 
let us suppose item (d) is not true. By definition, it means that
\begin{equation}\label{covering T2}
\bigcup_{K>0}B_K=\Bbb T^2.
\end{equation} 
Then,
 we will discuss all 
the possible cases under this assumption.
In every case, either item (a) is true, 
or item (b) is true, or item (c) is true, or it leads to a contradiction 
with the assumption (\ref{unbounded assumption}) or (\ref{covering T2}), 
which will conclude 
the proof.

Note $U$ is $f$-invariant, and $U$ is also open because it is a union of open sets. 
 We claim that $U$ is also dense.
 Suppose for contradiction that 
 $U$ is not dense, then there is some open set $W$ 
such that $W\bigcap U=\emptyset$. 
 It follows $W\bigcap B_K$ has no interior in $W$ 
for any $K>0$. On the other hand, since $W$ as an open subset of $\Bbb T^2$ is a Baire space, 
by Baire category theorem, we have $(\bigcup_{K>0}B_K) \cap W$ has no interior in $W$. 
This is clearly a contradiction to (\ref{covering T2}). This shows the claim.

If $U$ is essential, there are two subcases. Either $U$ is essential and not fully essential, or 
$U$ is fully essential. 

If $U$ is essential and not fully essential, 
then it has a connected component $U'$ that includes some loop representing an element 
in the first homology group $\text{H}_1(\Bbb T^2)\simeq \Bbb Z^2$, 
which is identified with 
some rational direction $(p,q)$. 
We can call this direction the homological direction of $U'$. 
Any connected component $\widetilde U'$ of $\pi^{-1}(U')$ 
by translation by $(p,q)$, but it must also be invariant by $\widetilde f$, otherwise every point in the rotation set of $\widetilde f$ would have a non-zero projection in the direction perpendicular to $(p,q)$. But since $\widetilde U'$ is invariant, then $\widetilde f$ have bounded deviation along the perpendicular direction of $(p,q)$.
It follows from assumption~(\ref{unbounded assumption}) that,
$v$ is not perpendicular to $(p,q)$. One can then consider $\mathrm{Fill}(U')$, the filling of $U$, obtained by taking the union of $U$ and all inessential connected components of its complement. Then $\mathrm{Fill}(U')$ is the annulus in  item (b) of the Proposition.

If $U$ is fully essential, then for any $z\notin U$, there exists some topological disk $D$ containing $z$, 
such that $\partial D\subset U$. 
Since $\partial D$ is compact, it is included in $I_K$ some sufficiently large $K>0$. 
It follows that 
the whole disk $D$ is included in $I_{K+\text{diam}(D)}$, where $\text{diam}_(D)$ denotes
the diameter of some connected component of the lift 
$\pi^{-1}(D) \subset \Bbb R^2$. In particular, 
$z\in I_{K+\text{diam}(D)} \subset U$, 
which is a contradiction. 
This contradiction shows $U=\Bbb T^2$ if $U$ is fully essential. 
However, since $\Bbb T^2$ is compact, 
it follows $\Bbb T^2=I_K$ for some large $K$, and 
this is a contradiction with the assumption (\ref{unbounded assumption}).
 
 We are left now to 
 the case when $U$ is inessential. Since $f$ is non-wandering, 
 this means $U$ is 
 a union of $f$-periodic disks. 
 Note that the assumption that case $(c)$ does not happen implies that, 
 the rotation set $\rho(\widetilde f)$ is contained in the closed half plane whose boundary is perpendicular to $v$. 
 By Corollary~\ref{Tal_noncontractible},
 there are several subcases.
 
 Subcase (1). %For some positive iterate $n>0$,
 $\pi(\text{Fix}(\widetilde f))$ is essential. This gives 
 item (a) of the proposition. 
 
 Subcase (2). 
 $f$ admits non-contractible periodic orbits. This is item (c).
% In this case, the rotation set $\rho(\widetilde f)$ cannot have non-empty interior. 
% Suppose otherwise, then Lemma~\ref{Guelman} with the first set of conditions implies $\widetilde f$ has bounded deviation, contradicting condition (\ref{unbounded assumption}). 
 
% So, $\rho(\widetilde f)$ must be a non-trivial rational segment. And either the non-contractible periodic points have all rotation vectors which are perpendicular to $v$, in which case  $\widetilde f$ admits  bounded deviation along the direction $v$, a contradiction with  condition (\ref{unbounded assumption}), or there exists some periodic point whose rotation vector.  By Remark~\ref{Pablo_annular}, in this case,  $f$ is annular, then either item (b) of the proposition holds, or   the rotation vector of all the non-contractible periodic point is  perpendicular to $v$. So the rotation set $\rho(\widetilde f)$  is a non-trivial rational segment. By Lemma~\ref{Guelman} with the second set of conditions,  $\widetilde f$ admits  bounded deviation along the direction $v$, a contradiction with  condition (\ref{unbounded assumption}).

 Subcase (3). There exists some $M>0$, such that for any point $x\in U$ and any lift $\widetilde x$ of $x$, 
 the orbit $\{\widetilde f^n ( \widetilde x)\}_{n\in \Bbb Z}$
 has diameter uniformly bounded by $M$ from above. 
 Then, since $U$ is dense, it follows for all point $x\in \Bbb T^2$, the same holds true. 
 This is a contradiction with assumption (\ref{unbounded assumption}). 
\end{proof}

\begin{Remark}
In the case of annular dynamics, an interesting 
example was provided in Proposition 1.1 in \cite{Fabio_Conejeros}, where 
the dynamics admits unbounded displacement in the negative direction, but there is no single orbit which is unbounded in the negative direction. In our context, this means it is possible that item (b) happens and item (c) and (d) do not happen. The example is not non-wandering, but satisfies a weaker condition, called the curve intersection property. 
There are also examples of torus homeomorphisms with unbounded deviation along every direction,
which satisfies 
item (a) of the proposition and does not satisfy items (b) and (c). 
See Theorem 3 of~\cite{Koro_Tal_irrotational}.
\end{Remark}
%\marginpar{\textcolor{blue}{This section is very good}}
\section{Non-contractible Periodic Orbits and Bounded Deviations}
This section is mainly devoted to proving Theorem~\ref{main_thm}, which will be split into 
two sub-sections. Finally, the proof of Theorem~\ref{rewritten_thm} is at the end of the section. 
Recall the conclusion of Theorem~\ref{main_thm} claims about the bounded deviation of certain homeomorphism. 
We will prove the theorem by contradiction. Thus, throughout, let us assume 
$\widetilde f\in 
\widetilde{\text{Homeo}}_{0,\text{nw}}^+(\Bbb T^2)$, 
whose rotation set $\rho(\widetilde f)$ is 
the line segment connecting 
$(0,0)$ and some irrational vector $(\alpha,\beta)$ so that $\alpha/\beta$ is also an irrational number. 
We will also fix a maximal isotopy $I$ for $f$ that lifts to a maximal isotopy $\widetilde I$ for $\widetilde f$. 
Then we find $\mathcal{F}$, $\widetilde{\mathcal{F}}$ which are Brouwer-Le Calvez foliations for $I$ and $\widetilde I$ respectively.

Assume, for a contradiction, that $\widetilde f$ admits unbounded deviations along the direction $-(\alpha,\beta)$.
\subsection{Unbounded trajectories and Bounded Leaves}
The first lemma provides abundance of points whose orbits are unbounded in the direction $-(\alpha,\beta)$.

\begin{lemma}\label{unbounded_orbits_to_the_left}
There exists some point $\widetilde x_0\in \Bbb R^2$ whose positive half-orbit $\{
\widetilde f^{n}(\widetilde x_0)\}_{n\geq 0}$
is unbounded in the direction $-(\alpha,\beta)$. 
Moreover, the set of points satisfying this property is dense in $\text{Ess}(f)$.
\end{lemma}

%\begin{coro} 
%Suppose $\widetilde f\in \widetilde {\text{Homeo}}^+_{0,\text{nw}}(\Bbb T^2)$
%whose rotation set $\rho(\widetilde f)$ is a line segment from $(0,0)$
%to some totally irrational vector $(\alpha,\beta)$. Assume 
%$\widetilde f$ has unbounded deviation along the direction $-(\alpha,\beta)$, 
%then there exists some point $\widetilde x$, such that the 
%$\widetilde f$-iterates of $\widetilde x$ has unbounded deviation along the direction $-(\alpha,\beta)$.
%\end{coro}

\begin{proof}
Take $v=-(\alpha,\beta)$. By Proposition~\ref{unbounded_trajectory},  we only need to show with respect to the direction $v$,
item (a), item (b) or item (c) do not happen. 
Note that 
items (b) and (c) imply the rotation set contains a 
line segment with rational slope, which is contradictory to the assumption on 
the shape of $\rho(\widetilde f)$. So they can not happen. 

Let us suppose for contradiction that item (a) in Proposition~\ref{unbounded_trajectory} is true. It means that $\text{Fix}(f)$ is not contained in a topological disk. Then, its complement $U$ either has 
a connected component which has a non-trivial homological direction, 
or all of its connected components are disks. 
If the first case happens, the rotation set must be contained in 
a line with rational slope, which is a contradiction. 
If the second case happens, 
by \cite{MZ}, we can consider 
a typical point $x$ for the ergodic measure $\mu$ whose average rotation number is $(\alpha,\beta)$. 
Such a typical point $x$ is $f$-recurrent. Moreover, 
\begin{equation}
\rho(\widetilde f,x)=(\alpha,\beta).
\end{equation}

Clearly, 
$x$ is contained in some 
connected component $D$ 
of $U$, which 
is a topological disk. Up to raising to some power of 
the dynamics, 
we can assume that some connected component 
$\widetilde D$ of $\pi^{-1}(D)$
is $\widetilde f$-invariant. Note $\widetilde D$ contains some lift 
$\widetilde x$.
Then we can choose a small $\delta>0$ such that 
the ball $B(\widetilde x,\delta)\subset \widetilde D$. 
Since $x$ is recurrent, there is an integer sequence
 $\{n_k\}_{k\geq 1}$, so that $f^{n_k}(x)\in B(x,\delta)$ for all $k$.
 Then, 
 \begin{equation}
 \widetilde f^{n_k}(\widetilde x) \in \underset{(p,q)\in \Bbb Z^2}{\bigcup} \big( B(\widetilde x,\delta)+(p,q) \big).
 \end{equation}
However, for any $(p,q)\in \Bbb Z^2\backslash \{(0,0)\}$, 
the set $B(\widetilde x,\delta)+(p,q)$
does not intersect $\widetilde D$ and therefore 
$\widetilde f^{n_k}(\widetilde x)$ does not belong to any of them. 
Therefore, $\widetilde f^{n_k}(\widetilde x) \in B(\widetilde x,\delta)$ for all $k$,
 and this implies 
$(\alpha,\beta)=0$, which is a contradiction. 
This completes the proof of the first assertion. 

Suppose there is some open set 
$U$ intersecting  $\text{Ess}(f)$, consisting of points for whose lift the 
positive half-orbit is bounded in the direction $-(\alpha,\beta)$. 
By definition of $\text{Ess}(f)$, 
\begin{equation}
W=\underset{n\in \Bbb Z}{\bigcup}f^n(U)
\end{equation} 
is essential. It follows for every point in the essential set $W$, its lift has positive half-orbit 
is bounded in the direction $-(\alpha,\beta)$. 
 Note the essential set $W$ is either fully essential, or it contains some annular component. 
 The second case can be excluded, since  as argued before, 
 it will lead to a contradiction with 
 the shape of $\rho(\widetilde f)$. 
We are left with the case when $W$ is fully essential, which means its complement is 
  inessential. Now there are two subcases. 
  
  In subcase one, there is a uniform bound $R$, 
  such that, for all point in $W$, the $\widetilde f$-positive half-orbit of its lift 
  is bounded in the direction $-(\alpha,\beta)$ by $R$. 
  In this case, 
  for any point $x\in W^c$, there is some 
   topological disk $D$ containing $x$, whose boundary circle $\ell$ is included in $W$. 
   Choose $\widetilde x$ and denote by $\widetilde D$ the connected component of $\pi^{-1}(D)$ which contains $\widetilde x$, 
   with boundary circle $\widetilde \ell$. 
   It follows that, the positive half-orbit of $\widetilde x$ is bounded in the direction  $-(\alpha,\beta)$
   by the uniform bound $R + \text{diam}(D)$.
  This shows for every point $\widetilde x \in \Bbb T^2$, its $\widetilde f$-positive half-orbit 
 bounded in the direction $-(\alpha,\beta)$, a contradiction to the assumption.
 
 In subcase two, $\widetilde f$ restricted to $\pi^{-1}(W)$ has unbounded deviation along the direction $-(\alpha,\beta)$. In this case, the beginning paragraphs in the proof of Proposition~\ref{unbounded_trajectory} shows that for a dense $G_\delta$ subset 
 of $W$, the positive $\widetilde f$-half orbit of their lifts are all unbounded in the direction of $-(\alpha,\beta)$, 
 which is a contradiction to the assumption.  
 So we have shown the second assertion. 
 \end{proof}
 
We also consider unbounded forward orbits in the direction $(\alpha,\beta)$ next. 
 \begin{lemma}\label{limsup_is_large}
Define 
 \begin{equation}\label{Q+definition}
Q :=\{ x \in \Bbb T^2 \big| 
 \{ \text{pr}_{(\alpha,\beta)} \big( \widetilde f^n(\widetilde x)-\widetilde x) \}_{n\geq 0} \text{ is unbounded to the right} \}.
 \end{equation}
Then $Q$ is dense in $\text{Ess}(f)$.
\end{lemma}
\begin{proof}
This proof is similar with the previous lemma. Note by the shape of the rotation set $\rho(\widetilde f)$, the following holds. 
\begin{equation}
\sup_{\widetilde x\in \Bbb R^2,n\geq 0} \text{pr}_{(\alpha,\beta)} ( \widetilde f^n(\widetilde x)-\widetilde x )=\infty.
\end{equation}
Then the rest argument goes exactly as in the proof of previous lemma. 
\end{proof}

Finally, we show all the leaves of $\mathcal F$ has uniformly bounded diameter. 
%\textcolor{black}{THIS LEMMA NEEDS  WORK!!!}
\begin{lemma}\label{bounded_leaves} There exists $M_{\widetilde{\mathcal F}}>0$ such that every leaf of $\widetilde{\mathcal F}$ has diameter smaller than $M_{\widetilde{\mathcal F}}$.
\end{lemma}
\begin{proof}
Write the "first quadrant" $Q_1:= \{w\in \Bbb R^2 \big| \langle w, v\rangle >0 \text{ and } 
\langle w, v^\perp \rangle >0 \}$ consisting of vectors with positive projections in both $v$ and $v^\perp$. 
Write the  "fourth quadrant" 
$Q_4=\{w\in \Bbb R^2\big| \langle w, v\rangle >0 \text{ and } 
\langle w, v^\perp \rangle <0 \}$. Lemma 2.1 of \cite{Guilherme_Thesis} says that one can find two $\widetilde{\mathcal{F}}$-transverse paths, $\alpha_1, \alpha_2:[0,1]\to \mathrm{Dom}(\widetilde{\mathcal{F}})$ such that $\alpha_1(1)=\alpha_1(0)+ w_1$, $\alpha_2(1)=\alpha_1(0)+ w_2$, where $w_1\in \Bbb Z^2\cap Q_1$ and $w_2\in \Bbb Z^2\cap Q_4$.

Recall the contradiction assumption says that there exists some point $\widetilde x_0$ whose positive half-orbit is unbounded in the $(-\alpha,\beta)$ direction. 
A standard compactness argument shows that one can find 
a point $\widetilde z$ which is not fixed by $\widetilde f$, 
an increasing sequence of $(n_k)_{ k\in \Bbb N}$ of positive integers, and a sequences of vectors $(u_k)_{k\in\Bbb N}$ in $\Bbb Z^2$, such that $\lim_{k\to\infty} \widetilde f^{n_k}(\widetilde x_0)-u_k=\widetilde z$, and such that $\lim_{k\to\infty}\mathrm{Pr}_{-(\alpha, \beta)}(u_k)=\infty$. 
By Lemma~\ref{Guilherme}, the projections of the $u_k$ in the direction perpendicular to $(\alpha, \beta)$ are uniformly bounded by a constant $L_0$.

We get, using Lemma~\ref{lemma_of_continuity}, that if $k_1<k_2$ are sufficiently large, 
then the path
$\widetilde I^{[n_{k_1}-1, n_{k_2}+1]}_{\widetilde{\mathcal{F}}}(\widetilde x_0)$ contains a subpath that is 
$\widetilde {\mathcal F}$-equivalent to an $\widetilde{\mathcal{F}}$-transverse path joining 
$\widetilde z+ u_{k_1}$ to $\widetilde z+ u_{k_2}$.  Note that, there exists some $L_1>0$ such that, if a vector $v$ is such that $\text{pr}_{-(\alpha,\beta)}(v)>L_1$ while $\vert\mathrm{pr}_{(\beta,-\alpha)}(v)\vert<L_0$, then the triangle whose vertices are $v, w_1$ and $w_2$ has the origin in its interior.  
Clearly, if we choose $k_1$ large and $k_2 \ll k_1$, 
then $\text{pr}_{-(\alpha,\beta)}(u_{k_2}-u_{k_1})>L_1$. This way, we find a third transverse path $\alpha_3:[0,1]\to \mathrm{Dom}(\widetilde{\mathcal{F}})$ such that $\alpha_3(1)=\alpha_3(0)+ w_3$ and such that the triangle whose vertices are $w_1, w_2$ and $w_3$ has the origin in its interior. This implies that there are positive integers $s_1,s_2,s_3$, such that $\sum_{i=1}^3s_i w_i=0$. Finally, one can apply Lemma~\ref{gradient-like} to conclude that
the diameter of every leaf of $\widetilde{\mathcal F}$ is bounded by some uniform constant $M_{\widetilde {\mathcal F}}$. 
\end{proof}

\subsection{Setting the Barrier and End of Proof}
In the rest of the argument, the main idea is to
put an infinite strip as a barrier, which is bounded in the horizontal direction. Then 
we use Lemma~\ref{unbounded_orbits_to_the_left} and Lemma~\ref{limsup_is_large} 
alternatively to construct ''oscilating" isotopy path, which passes the barrier as many times as we like, 
which will eventually become too crowded and satisfy the conditions for Proposition~\ref{new_prop}.

We need the following result, whose proof is exactly equal to the proofs of Lemma~\ref{local_good_neighbourhood} and the end of the proof of Theorem~~\ref{non-wandering_Bounded_M} and therefore is omitted.
\begin{lemma}\label{corta5}
There exists some $N_\ast>0$ and a finite set of leaves $\{\widetilde \phi_1, \widetilde\phi_2, \hdots, \widetilde\phi_k\}$ of $\widetilde{\mathcal{F}}$ such that, if $n>0$ and 
$\Vert{\widetilde f^n}(\widetilde x)-\widetilde x)\Vert>N_\ast$, then there exists some 
$j\in\{1, \hdots, k\}$ and five different vectors $w_i \in \Bbb Z^2$ for $1\le i\le 5$, 
such that $\widetilde I^n_{\widetilde{\mathcal{F}}}(\widetilde x)$ intersects $\widetilde \phi_j+w_i, 1\le i\le 5$.
\end{lemma}

Next, we show that there are trajectories that largely ``oscilate´´ in the horizontal direction as many times as necessary:

\begin{lemma}\label{oscila}
Given any $L>0, \varepsilon>0$ and $\widetilde x_0 \in (-1,0)^2$ which is a lift of 
an $f$-essential point $x_0$, there exists some $\widetilde y_L$ with $\Vert\widetilde x_0 -\widetilde y_L\Vert<\varepsilon$ and integers $n_i, m_i, 1\le i\le L$  with $n_i<m_i<n_{i+1}$ and such that:
\begin{itemize}
\item $\widetilde f^{n_i}(\widetilde y_L)$ is in $(-\infty, 0)\times \Bbb R$ and $\widetilde f^{m_i}(\widetilde y_L)$ is in $(N_\ast, \infty)\times \Bbb R$;
\item For $n_i<n<m_i$, $\widetilde f^{n}(\widetilde y_L)$ is in $[0, N_\ast]\times \Bbb R$.
\end{itemize}
\end{lemma}
\begin{proof}
The proof is elementary so we will be a bit sketchy. 
Fix $\varepsilon>0$. We look for each $y_L$ inductively. 
We can assume that $\varepsilon$ is sufficiently small 
such that $B_{\varepsilon}(\widetilde x_0)\subset (-1,0)^2$.
For $L=1$, note that,  by Lemma~\ref{limsup_is_large}, we can find an essential point $y_1$ which is 
$\varepsilon/2$ close to $\pi(\widetilde x_0)$ such that, if $\widetilde y_1$ is a lift of $y_1$ 
that is $\varepsilon/2$ close to $\widetilde x_0$, then  
$\text{pr}_{(\alpha,\beta)} \left( \widetilde f^n(\widetilde y_1)-\widetilde y_1\right)$ is unbounded from above.  
The existence of $n_1$ and $m_1$ follows easily from here. 
Now, given an integer $L$, a point $\widetilde y_L$ that projects to an essential point and that is at distance at most $(1-2^{-L})\varepsilon$ of $\widetilde x_0$ and the sequences of integer $n_i, m_i$, with $1\le i\le L$, 
satisfying the desired properties, we first find, using Lemma~\ref{unbounded_orbits_to_the_left} some point $\widetilde z_L$, projecting to an essential point, and having a future orbit unbounded in the $-(\alpha, \beta)$ direction. 
Furthermore, we can assume $\widetilde z_L$ is sufficiently close to $\widetilde y_L$ so that $\widetilde f^{n_i}(\widetilde z_L)$ is in $(-\infty, 0)\times \Bbb R$ and $\widetilde f^{m_i}(\widetilde z_L)$ is in $(N_\ast, \infty)\times \Bbb R$, for $1\le i\le L$. There exists some integer $s_L>m_L$ such that $\widetilde f^{s_L}(\widetilde z_L)$ is in $(-\infty, 0)\times \Bbb R$. We get the point $\widetilde y_{L+1}$ by applying again  Lemma~\ref{limsup_is_large} to $\widetilde z_L$ and finding a point whose orbit follows that of $\widetilde z_L$ sufficiently close until at least the iterate $s_L$, but for which the forward orbit is unbounded in the direction $(\alpha, \beta)$.
The induction step finishes and the conclusion follows because this induction process never ends. 
\end{proof}

\begin{lemma}\label{valordeM}
There exists some $M>0$ so that the following holds. 
For any $\widetilde z\in [-1,0]^2$, and any $1\le j\le k$, suppose there are $\overline{n}<\overline{m}$ such that,\begin{itemize}
\item $\widetilde f^{\overline{n}}(\widetilde z)$ is in $(-\infty, 0)\times \Bbb R$ and $\widetilde f^{\overline{m}}(\widetilde z)$ is in $(N_\ast, \infty)\times \Bbb R$;
\item For $\overline{n}<n<\overline{m}$, $\widetilde f^{n}(\widetilde z)$ is in $[0, N_\ast]\times \Bbb R$.
\item 
$\widetilde I^{[\overline{n}, \overline{m}]}_{\widetilde{\mathcal{F}}}(\widetilde z)$ intersects $\widetilde \phi_j+w$ with $w\in \Bbb Z^2$.
\end{itemize}
Then $\Vert w\Vert<M$.
\end{lemma}
\begin{proof}
Note first that, by Lemma~\ref{Guilherme}, there exists some constant $M_1$ such that the whole orbit of any point $\widetilde z$ that lies in $[-1,0]^2$ lies in a strip whose projection into a direction perpendicular to $(\alpha,\beta)$ has value less then $M_1$. This, and the hypotheses on $\overline{n}$ and $\overline{m}$ imply that there exists a compact set $K_1$ that contains $\widetilde f^n(\widetilde z), \overline{n}\le n\le \overline{m}$ if $\widetilde z \in
[-1,0]^2$, $\overline{n}$ and $\overline{m}$ are as in the hypothesis. In particular, from the continuity of the isotopy $\widetilde I$, we have that there exists a compact set $K\supset K_1$ such that, if $\widetilde y$ lies in $K_1$, then $\widetilde{I}^{[0,1]}(\widetilde y)$ lies in $K$. 
We deduce that $\widetilde{I}^{[\overline{n}, \overline{m}]}(\widetilde z)$ is contained in $K$.

Furthermore, since $\widetilde I^{[\overline{n}, \overline{m}]}_{\widetilde{\mathcal{F}}}(\widetilde z)$ is homotopic with fixed endpoints in $\mathrm{Dom}(\widetilde{\mathcal{F}})$ to $\widetilde{I}^{[\overline{n}, \overline{m}]}(\widetilde z)$ and since $\widetilde I^{[\overline{n}, \overline{m}]}_{\widetilde{\mathcal{F}}}(\widetilde z)$  is transverse to $\widetilde{\mathcal{F}}$, one has that every leaf that is intersected by 
$\widetilde I^{[\overline{n}, \overline{m}]}_{\widetilde{\mathcal{F}}}(\widetilde z)$  must also be intersected by $\widetilde{I}^{[\overline{n}, \overline{m}]}(\widetilde z)$. 
But since the leaves of $\widetilde{\mathcal{F}}$ are uniformly bounded, 
and only a finite number of integer translates of each $\widetilde \phi_j$ can intersect $K$, and the result follows.
\end{proof}

\begin{proof}[End of Proof of Theorem~\ref{main_thm}] 
This is now a simple application of Pigeon's hole principle.
The total number of sequences of five different elements in $\Bbb Z^2$ 
that have norm smaller than $M$ given by Lemma~\ref{valordeM} is upper bounded by some constant $k_1$. 
Pick $\widetilde x_0$ as in Lemma~\ref{oscila}, and pick $\varepsilon$ so that $B_{\varepsilon}(\widetilde x_0)\subset (-1,0)^2$. Pick $L>k_1.k+1$, where $k$ is given by Lemma~\ref{corta5}, and pick $\widetilde y_L$, $n_i, m_i, 1\le i\le L$ again as in Lemma~\ref{oscila}. We deduce that there exists some $1\le j_0\le k$ and a sequence of five distinct $\Bbb Z^2$ vectors $w_1, w_2, w_3, w_4, w_5$ and some $1\le i_1<i_2\le L$ such that both $\widetilde I^{[n_{i_1}, m_{i_1}]}_{\widetilde{\mathcal{F}}}(\widetilde y_L)$ 
and 
$\widetilde I^{[n_{i_2}, m_{i_2}]}_{\widetilde{\mathcal{F}}}(\widetilde y_L)$ intersect, in order, $\phi_{j_0}+w_i, 1\le i\le 5$. The theorem then follows by applying Proposition ~\ref{new_prop} to a re-parametrization of  
$\widetilde I^{[n_{i_1}, m_{i_2}]}_{\widetilde{\mathcal{F}}}(\widetilde y_L)$. 
\end{proof}

\begin{proof}[Proof of Theorem~\ref{rewritten_thm}] 
If $f$ is Hamiltonian the result follows directly from Lemma~\ref{hamiltonianlemma}. Assume then that $f$ is not Hamiltonian. 
By assumption, $\widetilde f$ admits fixed point and thus the rotation set $\rho(\widetilde f)$ contains $(0,0)$. Recall that $\rho(\widetilde f)$ is convex compact subset of $\Bbb R^2$. If it has interior, then it contains some rational vector in its interior. Therefore, 
by Lemma~\ref{franks1989realizing} $f$ admits non-contractable periodic orbit, which is excluded by our assumption. A similar argument, using the fact that $f$ preserves area and Lemma~\ref{lemmafranks2}, shows that $\rho(\widetilde f)$ cannot be a segment with two rational points, therefore it is either just the origin or a line segment with irrational slope containing $(0,0)$. Since we have also assumed that $f$ is not a non-Hamiltonian, the possibility that $\rho(\widetilde f)=\{(0,0)\}$ is also excluded.
By Lemma~\ref{ThmC_Forcing}, $\rho(\widetilde f)$ must be a segment whose two endpoints are $(0,0)$ and some vector $(\alpha,\beta)$ with $\alpha/ \beta$ irrational. Then, Theorem~\ref{main_thm} and Lemma~\ref{Guilherme} 
together show that 
$\widetilde f$ has a strong irrational dynamical direction $(\alpha,\beta)$.
\end{proof}

\section{Acknowledgements}
X-C. Liu is supported by 
Fapesp P\'os-Doutorado grant (Grant Number
 2018/03762-2). The authors thank Salvador Addas-Zanata for many helpful conversations. 
\bibliographystyle{plain}
\addcontentsline{toc}{chapter}{Bibliography}
\bibliography{BDeviation}

\begin{thebibliography}{10}

\bibitem{Salvador_Liu}
Salvador Addas-Zanata and Xiao-Chuan Liu.
\newblock On stable and unstable behaviour of certain rotation segments.
\newblock {\em arXiv preprint:1903.08703}, 2019.

\bibitem{Salvador_fareast}
Salvador Addas-Zanata and F{\'a}bio~Armando Tal.
\newblock On periodic points of area preserving torus homeomorphisms.
\newblock {\em Far East Journal of Dynamical Systems}, 9:371--378, 2007.

\bibitem{BCL_Fixed}
Fran{\c{c}}ois B{\'e}guin, Sylvain Crovisier, and Fr{\'e}d{\'e}ric Le~Roux.
\newblock Fixed point sets of isotopies on surfaces.
\newblock {\em Journal of the European Mathematical Society}, 22(6):1971--2046,
  2020.

\bibitem{Fabio_Conejeros}
Jonathan Conejeros and Fabio~Armando Tal.
\newblock Applications of forcing theory to homeomorphisms of the closed
  annulus.
\newblock {\em arXiv preprint arXiv:1909.09881}, 2019.

\bibitem{davalos2018annular}
Pablo D{\'a}valos.
\newblock On annular maps of the torus and sublinear diffusion.
\newblock {\em Journal of the Institute of Mathematics of Jussieu},
  17(4):913--978, 2018.

\bibitem{franks1989realizing}
John Franks.
\newblock Realizing rotation vectors for torus homeomorphisms.
\newblock {\em Transactions of the American Mathematical Society},
  311(1):107--115, 1989.

\bibitem{Franks2}
John Franks.
\newblock The rotation set and periodic points for torus homeomorphisms.
\newblock {\em arXiv preprint math/9605228}, 1996.

\bibitem{gurel2013non}
Ba{\c{s}}ak~Z G{\"u}rel.
\newblock On non-contractible periodic orbits of hamiltonian diffeomorphisms.
\newblock {\em Bulletin of the London Mathematical Society}, 45(6):1227--1234,
  2013.

\bibitem{Jaulent}
Olivier Jaulent.
\newblock Existence d'un feuilletage positivement transverse {\`a} un
  hom{\'e}omorphisme de surface.
\newblock {\em Ann. Inst. Fourier}, 64(4):1441--1476, 2014.

\bibitem{Koro_Tal_irrotational}
Andres Koropecki and Fabio Tal.
\newblock Area-preserving irrotational diffeomorphisms of the torus with
  sublinear diffusion.
\newblock {\em Proceedings of the American Mathematical Society},
  142(10):3483--3490, 2014.

\bibitem{Strictly_Toral}
Andres Koropecki and Fabio~Armando Tal.
\newblock Strictly toral dynamics.
\newblock {\em Inventiones mathematicae}, 196(2):339--381, 2014.

\bibitem{LC2005}
Patrice Le~Calvez.
\newblock Une version feuillet{\'e}e {\'e}quivariante du th{\'e}oreme de
  translation de {B}rouwer.
\newblock {\em Publications Math{\'e}matiques de l'Institut des Hautes
  {\'E}tudes Scientifiques}, 102(1):1--98, 2005.

\bibitem{Forcing}
Patrice Le~Calvez and Fabio~Armando Tal.
\newblock Forcing theory for transverse trajectories of surface homeomorphisms.
\newblock {\em Inventiones mathematicae}, 212(2):619--729, 2018.

\bibitem{Calvez_Tal_topological}
Patrice Le~Calvez and Fabio~Armando Tal.
\newblock Topological horseshoes for surface homeomorphisms.
\newblock {\em arXiv preprint arXiv:1803.04557}, 2018.

\bibitem{MZ}
Micha{\l} Misiurewicz and Krystyna Ziemian.
\newblock Rotation sets for maps of tori.
\newblock {\em Journal of the London Mathematical Society}, 2(3):490--506,
  1989.

\bibitem{orita2017existence}
Ryuma Orita.
\newblock On the existence of infinitely many non-contractible periodic orbits
  of hamiltonian diffeomorphisms of closed symplectic manifolds.
\newblock {\em arXiv preprint arXiv:1703.01731}, 2017.

\bibitem{Tal_noncontractible}
Fabio~Armando Tal.
\newblock On non-contractible periodic orbits for surface homeomorphisms.
\newblock {\em Ergodic Theory and Dynamical Systems}, 36(5):1644--1655, 2016.

\bibitem{Guilherme_Thesis}
Fabio~Armando Tal and Guilherme~Silva Salom{\~a}o.
\newblock Non-existence of sublinear diffusion for a class of torus
  homeomorphisms.
\newblock {\em Ergodic Theory and Dynamical Systems}, pages 1--31, 2021.

\end{thebibliography}
\end{document}